\documentclass{amsart}

\usepackage{amssymb,amsfonts,amscd}
\usepackage{enumerate}

\usepackage{hyperref}
\usepackage[nameinlink]{cleveref}

\theoremstyle{plain} 
\newtheorem{thm}{Theorem}[section]
\newtheorem{cor}[thm]{Corollary}
\newtheorem{prop}[thm]{Proposition}
\newtheorem{lem}[thm]{Lemma}

\theoremstyle{definition}
\newtheorem{defn}[thm]{Definition}

\newtheorem{exmps}[thm]{Examples}

\newtheorem*{ack*}{Acknowledgements}

\theoremstyle{remark}
\newtheorem{rem}[thm]{Remark}

\makeatletter
\let\c@equation\c@thm
\makeatother
\numberwithin{equation}{section}

\newcommand{\deriv}[2]{\frac{\partial #1}{\partial #2}}
\newcommand{\tanv}[1]{\frac{\partial}{\partial #1}}

\renewcommand{\tilde}{\widetilde}

\DeclareMathOperator{\End}{End}
\DeclareMathOperator{\Sym}{Sym}
\DeclareMathOperator{\tr}{tr}
\DeclareMathOperator{\Id}{Id}
\DeclareMathOperator{\av}{av}
\DeclareMathOperator{\Scal}{Scal}
\DeclareMathOperator{\ad}{ad}
\DeclareMathOperator{\Diff}{Diff}
\DeclareMathOperator{\GL}{GL}

\DeclareMathOperator{\Real}{Re}
\DeclareMathOperator{\orb}{orb}
\DeclareMathOperator{\ev}{ev}
\DeclareMathOperator{\ord}{ord}

\title{Orbifold Bergman Kernels}

\begin{document}

\begin{abstract}
Let $(\mathcal{X}, \omega)$ be a compact $n$-dimensional K\"ahler orbifold, the stabilizer groups of which are abelian and have rank at most two. Let $\mathcal{E}$ be an orbi-ample vector bundle of rank $2$ over $\mathcal{X}$ and let $H$ be a Hermitian metric on $\mathcal{E}$ such that the curvature form of $\det H$ is $-2\pi \sqrt{-1} \omega$. We show that a certain weighted sum of Bergman kernels for $\Sym^i \mathcal{E} \otimes \det(\mathcal{E})^{k+j}$ as $i$ and $j$ vary over a finite set admit an asymptotic expansion. This extends a similar result for cyclic K\"ahler orbifolds \cite{RT}. 

\end{abstract}

\author{Shin Kim}
\address{Department of Mathematics, Statistics, and Computer Science, University of Illinois Chicago, 
Chicago, IL 60607}
\email{skim651@uic.edu}

\author{Julius Ross}
\address{Department of Mathematics, Statistics, and Computer Science, University of Illinois Chicago, 
Chicago, IL 60607}
\email{juliusro@uic.edu}

\maketitle

\section{Introduction}

Let $(\mathcal{X}, \omega)$ be a compact $n$-dimensional K\"ahler orbifold with cyclic singularities and let $\mathcal{L}$ be an orbifold line bundle over $\mathcal{X}$ equipped with a Hermitian metric $h$ such that the curvature form of $h$ is $-2\pi \sqrt{-1} \omega$. The Hermitian metric $h$ and the K\"ahler form $\omega$ induce an $L^2$-inner product on the space of global sections $H^0(\mathcal{X}, \mathcal{L}^k)$. The Bergman function is the smooth function
\begin{equation*}
B_k(x) = \sum_\alpha \|s_\alpha(x)\|_{h^k(x)}^2,
\end{equation*}
where $\{s_\alpha\}$ is an orthonormal basis of $H^0(\mathcal{X}, \mathcal{L}^k)$.

In the case when $(\mathcal{X},\omega)$ is a compact K\"ahler manifold, the works of Tian \cite{Tian}, Zelditch \cite{Z}, Catlin \cite{Cat}, Ruan \cite{R}, Berman-Berndtsson-Sj\"ostrand \cite{BBS}, Dai-Liu-Ma \cite{DLM}, and Liu-Lu \cite{LL} show that the Bergman function admits an asymptotic expansion as $k$ tends to infinity. Moreover,  the first nontrivial coefficient in this asymptotic expansion equals half the scalar curvature of $\omega$, making the Bergman function a central tool in the study of K\"ahler metrics of constant scalar curvature (\cite{Don, Luo, Mab, PhS, Wang2, Yau, Zhang}). If $(\mathcal{X},\omega)$ is a K\"ahler orbifold, then the expansion of $B_k$ over the orbifold locus is not smooth (e.g.\ \cite{DLM}), but in \cite{RT} Ross-Thomas show that, on K\"ahler orbifolds with cyclic singularities, certain \emph{weighted sums} of the Bergman functions associated to different powers of $\mathcal{L}$ admit a global smooth expansion under the following mild condition.

\begin{defn} (\cite{FL}, \cite{RT})
Let $\mathcal{E}$ be an orbifold vector bundle over $\mathcal{X}$. Then, $\mathcal{E}$ is said to be \textit{locally ample} if for all points $x\in X$ the representation of the orbifold isotropy group $G_x$ on $\mathcal{E}_x$ is faithful. We say $\mathcal{E}$ is \textit{orbi-ample} if it is locally ample and the orbifold line bundle $\mathcal{L} = \text{det}(\mathcal E)^{\text{ord}(X)}$ descends to an ample line bundle on the underlying space of $\mathcal{X}$.
\end{defn}

\begin{thm} 
\textnormal{(\cite{RT}, Theorem 1.7)} \label{thm: cyclic} Let $(\mathcal{X}, \omega)$ be a compact $n$-dimensional K\"ahler orbifold with cyclic singularities and let $\mathcal{L}$ be an orbi-ample line bundle over $\mathcal{X}$. Let $h$ be a Hermitian metric on $\mathcal{L}$ such that the curvature form of $h$ is $-2\pi \sqrt{-1} \omega$. Fix nonnegative integers $N,p \geq 0$ and suppose that $\{c_i\}$ are finitely many positive constants chosen so that 
\begin{equation*}
\frac{1}{\ord(\mathcal{X})} \sum_i i^\ell c_i = \sum_{i \equiv u  \text{ \rm mod} \ord(\mathcal{X})} i^\ell c_i
\end{equation*}
for all $u$ and $\ell = 0, \dotsc, N+p$. Then, the function
\begin{equation*}
B_k^{\orb} := \sum_i c_i B_{k+i}
\end{equation*}
admits a global $C^p$-expansion of order $N$.
That is, there exist smooth functions $b_0, \dotsc, b_N$ on $\mathcal{X}$ such that
\begin{equation*}
B_k^{\orb} = b_0 k^{n} + b_1 k^{n-1} + \dotsb + b_N k^{n-N} + O(k^{n-N-1}),
\end{equation*}
where the $O(k^{n-N-1})$ term is to be taken in the $C^p$-norm. Furthermore, the $b_j$ are universal polynomials in the constants $c_i$ and the derivatives of the K\"ahler metric $\omega$. In particular
\begin{equation*}
b_0 = \sum_{i} c_i \hspace{10pt} \text{ and } \hspace{10pt} b_1 = \sum_i c_i \left ( ni + \frac{1}{2} \Scal_\omega \right ),
\end{equation*}
where $\Scal_\omega$ is the scalar curvature of $\omega$.

\end{thm}

For $k$ sufficiently large,
\begin{equation*}
\iota_k : \mathcal{X} \to \mathbb{P}\left ( \bigoplus_i H^0(\mathcal{X},\mathcal{L}^{k+i})^* \right )
\end{equation*}
is an embedding of $\mathcal{X}$ into a weighted projective space and $\iota_k^* \mathcal{O}(1) \cong \mathcal L$. Analogous to the case when $\mathcal{X}$ is a manifold, the asymptotic expansion of the orbifold Bergman functions has applications to the study of asymptotics of Fubini-Study metrics obtained by the embeddings $\iota_k$. For instance, \Cref{thm: cyclic} is used in \cite{RT2} to show that if $(\mathcal{X}, \mathcal{L})$ is a polarized orbifold with cyclic stabilizer groups and $\mathcal{X}$ admits an orbifold K\"ahler metric in $c_1(\mathcal{L})$ with constant scalar curvature, then $(\mathcal{X},\mathcal{L})$ is $K$-semistable.

Now, let $(\mathcal{X}, \omega)$ be a K\"ahler orbifold with stabilizer groups that are abelian and have rank at most two. In general it is not possible for the stabilizer groups to act faithfully on the fibers of an orbifold line bundle. So, we instead consider Bergman functions associated to symmetric powers $\Sym^i \mathcal{E}$ of an orbifold vector bundle $\mathcal{E}$ twisted by powers of the determinant bundle $\det(\mathcal{E})$ (this is the analytic analogy of the construction in \cite{FL}). 

Now the Bergman functions are endomorphisms of $\Sym^i \mathcal{E}$, where the powers $i$ may vary, and so they cannot directly be summed.  To form a suitable weighted sum we first take a ``trace" that is induced by the natural multiplication map that we describe now.

Let $V$ be a complex vector space and $$\mathfrak{a}_{i,k} : \End(V) \to \End(\Sym^i V)$$ denote the Lie algebra homomorphism associated to the multiplication map $$\Sym^i \otimes\,\, {\det}^{\otimes k} : \GL(V) \to \GL(\Sym^i V \otimes \det(V)^{\otimes k}) \cong \GL(\Sym^i V).$$
 We define $\tau_{i,k} : \End(\Sym^i V) \to \End(V)$ to be the dual to $\mathfrak{a}_{i,k}$ under the natural pairing, so by definition
\begin{equation} \label{eq: dual}
    \tr (\tau_{i,k}(A) B) = \tr  (A \mathfrak{a}_{i,k}(B)).
\end{equation}

Now let $\mathcal{E}$ be an orbi-ample vector bundle over $\mathcal{X}$ of rank $2$. Suppose that $H$ is a Hermitian metric on $\mathcal{E}$ such that $-2\pi \sqrt{-1} \omega$ is the curvature form of $h := \det(H)$, and set $\mathcal{L} = \det(\mathcal{E})$. We will denote by $$B_{i,k+j} \in \mathcal{A}^0(\End(\Sym^i\mathcal{E}))$$ the Bergman function for $(\Sym^i \mathcal{E} \otimes \mathcal{L}^{k+j}, \Sym^i H \otimes h^{k+j})$. Then, $\tau_{i,k+j} (B_{i,k+j})$ is an element of $\mathcal{A}^0(\End(\mathcal{E}))$ for any $i$, $j$, and $k$.

\begin{defn}
    Let $f$ be an analytic function in a neighborhood $U$ of $w \in \mathbb{C}^n$. Let $f(z) = \sum_{r=0}^\infty f_r(z)$ denote the power series expansion of $f$ at the point $w$ where $f_r$ is the $r$-th graded piece of the expansion. The function $f$ is said to have \emph{total order $r$ at the point $w$} if $f_r$ is the lowest degree term that does not vanish identically. 
\end{defn}

\begin{thm} \label{thm: main}
Fix nonnegative integers $N, p \geq 0$. Let $\zeta_{\ord(\mathcal{X})}$ be a primitive $\ord(\mathcal{X})$-th root of unity. Let $\{c_{ij}\}_{i,j \geq 0}$ be a set of nonnegative constants such that the polynomial
\begin{equation*}
\sum_{i,j \geq 0} c_{ij} \left ( \sum_{\mu \in (\mathbb{Z}_{\geq 0})^2, |\mu|=i} {T_1}^{\mu_1} {T_2}^{\mu_2} \right ){T_3}^j
\end{equation*} 
has total order at least $5N+5p+5$ at $(T_1,T_2,T_3) = ({\zeta_{\ord(\mathcal{X})}}^{m_1}, {\zeta_{\ord(\mathcal{X})}}^{m_2}, {\zeta_{\ord(\mathcal{X})}}^{m_1+m_2})$ for all $m_1$,$m_2$ $\in \mathbb{Z}$ such that $({\zeta_{\ord(\mathcal{X})}}^{m_1},{\zeta_{\ord(\mathcal{X})}}^{m_2})$ does not equal $(1,1)$. 

For all nonnegative integers $k \geq 0$, define $B_k^{\orb} \in \mathcal{A}^0(\End(\mathcal{E}))$ by
\begin{equation} \label{eq: orbifold Bergman function}
    B_k^{\orb}(x) = \sum_{i,j \geq 0} c_{ij}\tau_{i,k+j} \left ( B_{i,k+j}(x)  \right ).
\end{equation}
Then $B_k^{\orb}$ admits an asymptotic expansion of order $N$ in the $C^p$-norm.  That is, there exist smooth sections $b_0$, $\dotsc$, $b_N$ in $\mathcal{A}^0(\End(\mathcal{E}))$ such that
    \begin{equation*}
        B_k^{\orb} =  b_0 k^{n+1} + \dotsb + b_{N} k^{n-N+1} + O(k^{n - N})
    \end{equation*}
with respect to the $C^p$-norm. Moreover the $b_i$ are universal quantities in terms of $\omega$ and its derivatives; in particular
\begin{equation*}
        b_0(x) = \sum_{i,j \geq 0} c_{ij} (i+1) \Id_E
\end{equation*}
and
\begin{equation*}
    b_1(x) =  \sum_{i,j \geq 0} c_{ij} \left (nj + \frac{(i+1)(2i+1)}{2} \Scal_\omega(x) \right )\Id_E.
\end{equation*}
\end{thm}

There are many possible choices of $c_{ij}$ that satisfy the hypothesis of the theorem.  One of these is made by setting
$$c_{ij} =a_{i+j} a_j - a_{i+j+1}a_{j-1}$$
where the $a_i$ are defined by
$$\sum_i a_i t^i = (1+ t + \cdots t^d)^p.$$
These $c_{ij}$ are non-negative as the $a_i$ are log concave.

To prove Theorem \ref{thm: main}, we will adapt the arguments in \cite{RT}.  The key difference is that in this non-abelian case we need control over the form of the coefficients that appear in the expansion of the Bergman kernels $B_{i,k+j}$ of $\Sym^i \mathcal E\otimes \mathcal L^{j+k}$ for $k\gg 0$.  In \Cref{orbifolds}, we recall some basic facts about orbifolds. In \Cref{local reproducing kernels}, we will define local reproducing kernels as in \cite{BBS} and \cite{RT}, and show that the local reproducing kernels approximate the global Bergman kernels. In \Cref{coefficients}, we analyze the coefficients $b_{i,q}$ of the Bergman functions $B_{i,k+j}$ and show that they must take on a particular form. We will prove \Cref{thm: main} in \Cref{expansion}. Finally, we will show in \Cref{Schur} that constants $\{c_{ij}\}$ that satisfy the hypotheses of \Cref{thm: main} always exist and how they relate to Schur positivity of $(1+z+ \cdots+ z^d)^p (1+ w + \cdots + w^d)^p$.

\begin{ack*}
SK would like to thank Sayok Chakravarty and Henry Fontana for pointing out \Cref{lem: Log-concave and Schur}. We also thank Ruadha\'i Dervan, Stephen Mackes, Jaegeon Shin, Jennifer Vaccaro, and Edward Varvak for helpful conversations.  JR is partially supported by a Simons Travel Award.
\end{ack*}

\section{Orbifolds} \label{orbifolds}

In this section, we recall necessary definitions and facts about orbifolds and orbi-bundles. We refer the readers to \cite{B}, \cite{BG}, \cite{MM}, and \cite{S} for more details.

\begin{defn}
    Let $X$ be a topological space. An \emph{orbifold chart} of dimension $n \geq 0$ on $X$ is a triple $(U, G, \varphi)$ where $U$ is a connected open subset of $\mathbb{R}^n$, $G$ is a finite subgroup of diffeomorphisms of $U$, and $\varphi : U \to X$ is an open map that induces a homeomorphism $U/G \to \varphi(U)$.

    Let $(U,G,\varphi)$ and $(U',G',\varphi')$ be orbifold charts. An \emph{embedding} $\lambda : (U',G',\varphi') \to (U,G,\varphi)$ between orbifold charts is a smooth embedding $\lambda : U' \to U$ such that $\varphi \circ \lambda = \varphi'$.

    We say that an orbifold chart $(U,G,\varphi)$ is a \emph{holomorphic chart} of complex dimension $n$ if $U$ is a subset of $\mathbb{C}^n$ and $G$ is contained in the group of biholomorphisms of $U$. Similarly, we say that an embedding $\lambda : (U',G',\varphi') \to (U,G,\varphi)$ between orbifold charts is a \emph{holomorphic embedding} between orbifold charts if $\lambda: U' \to U$ is holomorphic.
\end{defn}

We recall some notation regarding group actions on a manifold. Let $M$ be a manifold and denote the group of diffeomorphisms of $M$ by $\Diff(M)$. Let $G$ be a finite subgroup of $\Diff(M)$. The \emph{isotropy group} of a subset $S \subseteq M$ is
\begin{equation*}
    G_S = \{g \in G : g \cdot S = S\}.
\end{equation*}
If $S = \{x\}$ only has one element, then we will write $G_x$ for $G_S$. A subset $S \subseteq M$ is called \emph{$G$-stable} if it is connected and, for any $g \in G$, either $g \cdot S = S$ or $g \cdot S \cap S = \emptyset$. The following proposition lists basic properties of embeddings between orbifold charts. The proof of the proposition can be found in \cite{MM}.

\begin{prop} \label{prop: orbifold}
\begin{enumerate}
    \item For any embedding $\lambda : (U',G',\varphi') \to (U, G, \varphi)$ between orbifold charts on $X$, the image $\lambda (U')$ is a $G$-stable open subset of $U$, and there is a unique isomorphism $\overline{\lambda} : G' \to G_{\lambda(U')}$ for which $\lambda(g' \cdot \tilde{x}) = \overline{\lambda}(g') \cdot \lambda(\tilde{x})$.
    
    \item The composition of embeddings between orbifold charts is an embedding between orbifold charts.
    
    \item For any orbifold chart $(U,G,\varphi)$, any diffeomorphism $g \in G$ is an embedding of $(U,G,\varphi)$ into itself and $\overline{g}(h) = ghg^{-1}$.

    \item If $\lambda, \mu : (U',G',\varphi') \to (U,G,\varphi)$ are two embeddings between the same orbifold charts, there exists a unique $g \in G$ with $\lambda = g \circ \mu$.
\end{enumerate} 
\end{prop}

\begin{defn}
    We say that two (holomorphic) orbifold charts $(U,G,\varphi)$ and $(U',G',\varphi')$ of (complex) dimension $n$ on $X$ are \emph{compatible} if, for any $x \in \varphi(U) \cap \psi(U')$, there exist an (holomorphic) orbifold chart $(W,K,\theta)$ on $X$ with $x \in \theta(W)$ and (holomorphic) embeddings between orbifold charts $\lambda : (W,K,\theta) \to (U,G,\varphi)$ and $\mu : (W,K,\theta) \to (U',G',\varphi')$.
    
An \emph{(complex) orbifold atlas} of dimension $n$ on a topological space $X$ is a collection $\mathcal{U} = \{(U_i,G_i,\varphi_i)\}_{i \in I}$ of pairwise compatible (holomorphic) orbifold charts of (complex) dimension $n$ on $X$ such that $\bigcup_{i \in I} \varphi_i(U_i) = X$. Two (complex) orbifold atlases of $X$ are equivalent if their union is an (complex) orbifold atlas.

    An \emph{(complex) orbifold} of dimension $n$ is a pair $\mathcal{X} = (X,\mathcal{U})$, where $X$ is a second-countable Hausdorff topological space and $\mathcal{U}$ is a maximal (complex) orbifold atlas of (complex) dimension $n$ on $X$.

\end{defn}

\begin{defn}
    Let $\mathcal{X} =(X,\mathcal{U})$ be an orbifold and let $x \in X$. Choose an orbifold chart $(U,G,\varphi) \in \mathcal{U}$ and $\tilde{x} \in U$ such that $\varphi(\tilde{x}) = x$. Then, the \emph{isotropy group of $x$} is $G_x := G_{\tilde{x}}$. By \Cref{prop: orbifold}, the isotropy group of $x$ does not depend on our choice of $(U,G,\varphi)$ and $\tilde{x} \in U$. 
\end{defn}

\begin{defn}
    A smooth real \emph{orbifold vector bundle $\mathcal{E}$} of rank $r$ over an orbifold $\mathcal{X} = (X,\mathcal{U})$ consists of a smooth real vector bundle $E_{U}$ of rank $r$ over $U$ for each orbifold chart $(U,G,\varphi) \in \mathcal{U}$ together with an anti-homomorphism $\rho_U$ from $G$ to the group of equivalences of $E_U$ onto itself satisfying the following conditions.
    \begin{enumerate}
        \item If $\xi$ lies in the fiber ${E_U}_{\mid x}$ over $x \in U$, then, for each $g \in G$, $\rho_U(g) (\xi)$ lies in the fiber over $g^{-1} \cdot x$. 

        \item If $\lambda : (U',G',\varphi') \to (U,G,\varphi)$ is an embedding between orbifold charts on $X$, then there is a bundle map $\lambda^* : {E_{U}}_{\mid \lambda(U')} \to E_{U'}$ such that, for any $g' \in G'$, 
        \begin{equation*}
            \rho_{U'} (g')  \circ \lambda^* = \lambda^* \circ \rho_{U}\left (\overline{\lambda}(g') \right ).
        \end{equation*}

        \item If $\lambda_{ji} : (U_i,G_i,\varphi_i) \to (U_j,G_j,\varphi_j)$ and $\lambda_{kj} : (U_j,G_j,\varphi_j) \to (U_k,G_k,\varphi_k)$ are embeddings between orbifold charts on $X$, then $(\lambda_{kj} \circ {\lambda_{ji}})^* = \lambda_{ji}^* \circ \lambda_{kj}^*$.
    \end{enumerate}

   Smooth complex orbifold vector bundles are defined in an analogous way. We say that a smooth complex orbifold vector bundle $\mathcal{E}$ over a complex orbifold $\mathcal{X} = (X,\mathcal{U})$ is \emph{holomorphic} if, for each $(U,G,\varphi) \in \mathcal{U}$, $E_U$ is a holomorphic vector bundle and, for any embedding $\lambda : (U',G',\varphi') \to (U,G,\varphi)$, $\lambda^*: {E_U}_{\mid \lambda(U')} \to E_{U'}$ is a holomorphic bundle map.
\end{defn}

\begin{defn}
Let $\mathcal{E}$ be an orbifold vector bundle over an orbifold $\mathcal{X} = (X,\mathcal{U})$. Then, a \emph{section} $\sigma$ of $\mathcal{E}$ over an open set $W \subseteq X$ consists of a section $\sigma_U$ of the bundle $E_U$ for each orbifold chart $(U,G,\varphi) \in \mathcal{U}$ with $\varphi(U) \subseteq W$ such that,
for any embedding $\lambda : (U',G',\varphi') \to (U,G,\varphi)$  between orbifold charts and any $\tilde{x} \in U'$,
\begin{equation*}
\lambda^* ( \sigma_U(\lambda(\tilde{x}))) = \sigma_{U'}(\tilde{x}).
\end{equation*}

If $\mathcal{E}$ is holomorphic and each of the local sections $\sigma_U$ is holomorphic, we say that $\sigma$ is a \emph{holomorphic section}.
\end{defn}

\begin{exmps}

Let $\mathcal{X} = (X,\mathcal{U})$ be a complex orbifold. 

\begin{itemize}
\item The \emph{tangent bundle} $\mathcal{TX}$ over an orbifold $\mathcal{X}$ is a real orbifold vector bundle comprised of the local tangent bundles $TU$ for each $(U,G,\varphi) \in \mathcal{U}$. The complexified tangent bundle $\mathcal{T}_\mathbb{C} \mathcal{X}$ splits into a direct sum of the holomorphic tangent bundle $\mathcal{T}^{(1,0)}\mathcal{X}$ and its conjugate $\mathcal{T}^{(0,1)}\mathcal{X}$.

\item $\bigwedge^p \mathcal{T^*X}$ are real orbifold vector bundles over $\mathcal{X}$ and $\bigwedge^{(p,q)} \mathcal{T}^*\mathcal{X} := \left ( \bigwedge^p \mathcal{T}^{(1,0)}\mathcal{X} \right ) \otimes \left ( \bigwedge^q \mathcal{T}^{(0,1)}\mathcal{X} \right )$ are complex orbifold vector bundles over $\mathcal{X}$

\item A \emph{smooth $p$-form} over $\mathcal{X}$ is a section of $\bigwedge^p \mathcal{T^*X}$. In other words, it consists of $p$-forms $\omega_{U}$ for each orbifold chart $(U,G,\varphi) \in \mathcal{U}$ such that if $\lambda : (U',G',\varphi') \to (U,G,\varphi)$ is an embedding between orbifold charts, then $\lambda^* \omega_U = \omega_{U'}$.

\item A \emph{smooth $(p,q)$-form} on $\mathcal{X}$ is a section of $\bigwedge^{(p,q)} \mathcal{T}^*\mathcal{X}$.

\item A \emph{K\"ahler form} $\omega$ on $\mathcal{X}$ consists of a $(1,1)$-form on $\mathcal{X}$ such that, for each orbifold chart $(U,G,\varphi) \in \mathcal{U}$, the corresponding local form $\omega_U$ is K\"ahler.
\end{itemize}
\end{exmps}
 
\begin{rem}

Let $\mathcal{X} = (X,\mathcal{U})$ be a complex orbifold. Let $(U,G,\varphi) \in \mathcal{U}$ be an orbifold chart and let $\alpha$ be a $2n$-form defined on $\varphi(U)$. If $\mathcal{O} \subseteq \varphi(U)$ is an open subset, then the integral of $\alpha$ over $\mathcal{O}$ is defined by
\begin{equation*}
\int_\mathcal{O} \alpha = \frac{1}{|G|} \int_{\varphi^{-1}(\mathcal{O})} \alpha_U.
\end{equation*}
Using this definition, we can define the integral of a $2n$-form over $\mathcal{X}$. If $\alpha$ is a smooth $(2n-1)$-form with compact support on $\mathcal{X}$, then $\int_{\mathcal{X}} d\alpha = 0$ (\cite{B}).
\end{rem}

\begin{defn}
Let $\mathcal{X} = (X,\mathcal{U})$ be a complex orbifold. A \emph{Riemannian metric} $g$ on $\mathcal{X}$ is a collection of Riemannian metrics $g_U$ for each orbifold chart $(U,G,\varphi) \in \mathcal{U}$ such that, for any embedding $\lambda : (U',G',\varphi') \to (U,G,\varphi)$, $\tilde{x} \in U'$, and $v,w \in {T_U}_{\mid \lambda(\tilde{x})}$,
\begin{equation*}
g_{U'}(\tilde{x}) (\lambda^*(v), \lambda^*(w)) = g_U(\lambda(\tilde{x})) (v,w).
\end{equation*}

Similarly, a \emph{Hermitian metric} $h$ on a complex orbifold vector bundle $\mathcal{E}$ over $\mathcal{X}$ is a collection of Hermitian metrics $h_U$ on $E_U$ for each orbifold chart $(U,G,\varphi) \in \mathcal{U}$ such that
\begin{equation*}
h_{U'}(\tilde{x}) (\lambda^*(\xi), \lambda^*(\eta)) = h_U(\lambda(\tilde{x})) (\xi,\eta).
\end{equation*}
for any embedding $\lambda : (U',G',\varphi') \to (U,G,\varphi)$, $\tilde{x} \in U'$, and $\xi,\eta \in {E_U}_{\mid \lambda(\tilde{x})}$.
\end{defn}

\begin{defn}
Let $\mathcal{X} = (X,\mathcal{U})$ be a complex orbifold and let $\mathcal{E}_1$ and $\mathcal{E}_2$ be orbifold vector bundles over $\mathcal{X}$. An \emph{orbifold vector bundle map} $\theta : \mathcal{E}_1 \to \mathcal{E}_2$ is a collection of vector bundle maps $\theta_U : (E_1)_U \to (E_2)_U$ for each orbifold chart $(U,G,\varphi) \in \mathcal{U}$ such that, for any injection $\lambda: (U',G',\varphi') \to (U,G,\varphi)$, the following diagram commutes.


\[
\begin{CD}
{(E_1)_U}_{\mid \lambda(U')} @>{\theta_U}>> {(E_2)_U}_{\mid \lambda(U')} \\
@V{\lambda^*}VV @VV{\lambda^*}V \\
(E_1)_{U'} @>{\theta_{U'}}>> (E_2)_{U'}
\end{CD}
\]

\end{defn}

\begin{exmps}

Let $\mathcal{X} = (X,\mathcal{U})$ be a complex orbifold and let $\mathcal{E}$ be a smooth complex orbifold vector bundle over $\mathcal{X}$. Let $\omega$ be a K\"ahler form on $\mathcal{X}$. Then, $\omega$ induces a Riemannian metric $g$ on $\mathcal{X}$.

\begin{itemize}
\item The \emph{Hodge $*$ operator} $*$ with respect to the Riemannian metric $g$ on $\left ( \bigwedge^{(p,q)} \mathcal{T}^*\mathcal{X} \right ) \otimes \mathcal{E}$ is the orbifold vector bundle map $* : \left ( \bigwedge^{(p,q)} \mathcal{T}^*\mathcal{X} \right ) \otimes \mathcal{E} \to \left ( \bigwedge^{(n-q,n-p)} \mathcal{T}^*\mathcal{X} \right ) \otimes \mathcal{E}$ that is the collection of the usual Hodge $*$ operators $*_U : \left ( \bigwedge^{(p,q)} T^*U \right ) \otimes E_U \to \left ( \bigwedge^{(n-q,n-p)} T^*U \right ) \otimes E_U$ for each orbifold chart $(U,G,\varphi) \in \mathcal{U}$.  

\item The \emph{Lefschetz operator} $L: \left ( \bigwedge^{(p,q)} \mathcal{T}^*\mathcal{X} \right ) \otimes \mathcal{E} \to \left ( \bigwedge^{(p+1,q+1)} \mathcal{T}^*\mathcal{X} \right ) \otimes \mathcal{E}$ and the \emph{dual Lefschetz operator} $\Lambda : \left ( \bigwedge^{(p,q)} \mathcal{T}^*\mathcal{X} \right ) \otimes \mathcal{E} \to \left ( \bigwedge^{(p-1,q-1)} \mathcal{T}^*\mathcal{X} \right ) \otimes \mathcal{E}$  are defined analogously.

\end{itemize}
\end{exmps}

\begin{defn}
Let $\mathcal{X} = (X,\mathcal{U})$ be an orbifold. An \emph{orbisheaf} on $\mathcal{X}$ consists of a sheaf $\mathcal{F}_U$ for each orbifold chart $(U,G,\varphi) \in \mathcal{U}$ such that the following conditions hold.

\begin{enumerate}
\item If $\lambda: (U',G',\varphi') \to (U,G,\varphi)$ is an embedding, then there is an isomorphism of sheaves $\mathcal{F}(\lambda) : \mathcal{F}_{U'} \to \lambda^* \mathcal{F}_{U}$. 
\item If $\lambda: (U',G',\varphi') \to (U,G,\varphi)$ and $\lambda' :  (U'',G'',\varphi'') \to (U',G',\varphi')$ are embeddings, then
\begin{equation*}
\mathcal{F}(\lambda \circ \lambda') = \mathcal{F}(\lambda') \circ \lambda'^* \mathcal{F}(\lambda).
\end{equation*}
\end{enumerate}

\end{defn}

\begin{exmps}

Let $\mathcal{X} = (X,\mathcal{U})$ be a complex orbifold and let $\mathcal{E}$ be a holomorphic orbifold vector bundle over $\mathcal{X}$.

\begin{itemize}
\item The orbisheaf $\mathcal{A}^k(\mathcal{E})$ of smooth orbifold sections of $\left (\bigwedge^k \mathcal{T}^* \mathcal{X} \right ) \otimes \mathcal{E}$ is the collection of sheaves $\mathcal{A}^k(E_U)$ of smooth differential $k$-forms with values in $E_U$ for each orbifold chart $(U,G,\varphi) \in \mathcal{U}$. 

\item Similarly, the orbisheaf $\mathcal{A}^{p,q}(\mathcal{E})$ of smooth orbifold sections of $\left (\bigwedge^{(p,q)} \mathcal{T}^* \mathcal{X} \right ) \otimes \mathcal{E}$ is the collections of sheaves $\mathcal{A}^{p,q}(E_U)$ for each orbifold chart $(U,G,\varphi) \in \mathcal{U}$.

\end{itemize}
\end{exmps}

\begin{defn}
A \emph{morphism of orbisheaves} $\theta : \mathcal{F}_1 \to \mathcal{F}_2$ is a family of sheaf maps $\theta_U : (\mathcal{F}_1)_U \to (\mathcal{F}_2)_U$ for each orbifold chart $(U,G,\varphi) \in \mathcal{U}$ such that, for any embedding $\lambda : (U',G',\varphi') \to (U,G,\varphi)$, the following diagram commutes.


\[
\begin{CD}
(\mathcal{F}_1)_{U'} @>{\theta_{U'}}>> (\mathcal{F}_2)_{U'} \\
@V{\mathcal{F}_1(\lambda)}VV @VV{\mathcal{F}_2(\lambda)}V \\
\lambda^*(\mathcal{F}_1)_U @>{\lambda^*\theta_U}>> \lambda^*(\mathcal{F}_2)_U
\end{CD}
\]

\end{defn}

Now, let $\mathcal{X} = (X,\mathcal{U})$ be a compact complex orbifold equipped with a K\"ahler metric $\omega$. Let $\mathcal{E}$ be a holomorphic orbifold vector bundle over $\mathcal{X}$ equipped with an orbifold Hermitian metric $H$. The \emph{Chern connection} $\nabla : \mathcal{A}^k(\mathcal{E}) \to \mathcal{A}^{k+1}(\mathcal{E})$ is the collection of Chern connections $\nabla_U : \mathcal{A}^k(E_U) \to \mathcal{A}^{k+1}(E_U)$ with respect to $H_U$ for each orbifold chart $(U,G,\varphi) \in \mathcal{U}$. 

\begin{lem} \label{lem: Chern connection}
The Chern connection $\nabla : \mathcal{A}^k(\mathcal{E}) \to \mathcal{A}^{k+1}(\mathcal{E})$ is a morphism of orbisheaves. 
\end{lem}

\begin{proof}
Let $\lambda : (U',G',\varphi') \to (U,G,\varphi)$ be an embedding between orbifold charts and set 
\begin{equation*}
\tilde{\nabla}_{U'} = \lambda^* \circ \nabla_U \circ (\lambda^*)^{-1}.
\end{equation*}  
We will show that $\tilde{\nabla}_{U'} = \nabla_{U'}$ on local sections of $E_{U'}$. This is sufficient to prove the lemma since the exterior derivative $d$ commutes with the pullback map $\lambda^*$ and, if $\alpha \otimes s$ is a local section of $\wedge^k T^*U' \otimes E_{U'}$ and $D$ is a connection, then
\begin{equation*}
D (\alpha \otimes s) = d \alpha \otimes s + (-1)^k \alpha \wedge D s.
\end{equation*} 

Let $W \subseteq U'$ be an open subset such that ${E_{U'}}_{\mid W}$ is a trivial vector bundle. Then, there exists a holomorphic matrix-valued function $\rho$ on $\lambda(W)$ such that, for any section $s$ of ${E_{U'}}_{\mid W}$ and for any $\tilde{x} \in W$,
\begin{equation*}
((\lambda^*)^{-1}  s) (\lambda(\tilde{x}) ) =  \rho(\lambda(\tilde{x}))^{-1}  s ( \tilde{x}).
\end{equation*}
Writing $\nabla_U : \mathcal{A}^0(E_U) \to \mathcal{A}^{1}(E_U)$ as $\nabla_U = d + A_U$, we observe that, for any section $s$ of $ {E_{U'}}_{\mid W}$,
\begin{equation} \tag{\textasteriskcentered} \label{eq: pullback connection}
\tilde{\nabla}_{U'} (s) = (\rho\circ \lambda) \left ( d ((\rho\circ \lambda)^{-1} s) \right ) + (\rho \circ \lambda) \lambda^*( A_U) (\rho \circ \lambda)^{-1} s.
\end{equation}
The equation \eqref{eq: pullback connection} and the fact that $\rho \circ \lambda$ is holomorphic imply that $\tilde{\nabla}_{U'}$ is compatible with the holomorphic structure. Moreover, for any sections $s_1$ and $s_2$ of ${E_{U'}}_{\mid W}$, we observe that
\begin{equation*}
\begin{split}
& H_{U'} (\tilde{\nabla}_{U'}  s_1,  s_2) + H_{U'} (s_1,\tilde{\nabla}_{U'} s_2) \\
& = \lambda^* H_{U} (\nabla_{U} (\rho^{-1} (s_1 \circ \lambda^{-1})), \rho^{-1} (s_2 \circ \lambda^{-1})) + \lambda^* H_{U} ( \rho^{-1} (s_1 \circ \lambda^{-1}), \nabla_U (\rho^{-1} (s_2 \circ \lambda^{-1}))) \\
& = \lambda^* d( H_U(\rho^{-1} (s_1 \circ \lambda^{-1}), \rho^{-1} (s_2 \circ \lambda^{-1}))) \\
& = d (H_{U'} (s_1, s_2)).
\end{split}
\end{equation*}
Therefore, $\tilde{\nabla}_{U'}$ must be the Chern connection.

\end{proof}

The \emph{curvature} with respect to the Chern connection $\nabla$ is defined by 
\begin{equation*}
F_H := \nabla \circ \nabla.
\end{equation*}
Then, $F_H: \mathcal{E} \to \left (\bigwedge^{(1,1)}  \mathcal{T}^*\mathcal{X}\right ) \otimes \mathcal{E}$ is an orbifold vector bundle map. The Chern connection has a decomposition 
\begin{equation*}
\nabla = \nabla^{(1,0)} + \overline{\partial}
\end{equation*}
where $\nabla^{(1,0)} : \mathcal{A}^{p,q}(\mathcal{E}) \to \mathcal{A}^{p+1,q}(\mathcal{E})$ is the $(1,0)$ part of $\nabla$. Denote by $\langle \cdot, \cdot \rangle$ the Hermitian metric on $\bigwedge^{p,q} \mathcal{T}^*X \otimes \mathcal{E}$ induced by $H$ and the K\"ahler metric. The induced $L^2$ inner product $( \cdot, \cdot) : \mathcal{A}^{p,q}(\mathcal{X}, \mathcal{E}) \times \mathcal{A}^{p,q}(\mathcal{X}, \mathcal{E}) \to \mathbb{C}$ is defined by 
\begin{equation*}
   (s,t) :=  \int_\mathcal{X} \langle s, t \rangle \frac{\omega^n}{n!}. 
\end{equation*}
The morphisms of orbisheaves ${\nabla^{(1,0)}}^*:\mathcal{A}^{p,q}(\mathcal{E}) \to \mathcal{A}^{p-1,q}(\mathcal{E})$ and $\overline{\partial}^*:\mathcal{A}^{p,q}(\mathcal{E}) \to \mathcal{A}^{p,q-1}(\mathcal{E})$ defined by 
\begin{equation*}
{\nabla^{(1,0)}}^* := - * \circ   \overline{\partial} \circ * \hspace{10pt} \text{ and } \hspace{10pt} \overline{\partial}^* := - * \circ   \nabla^{(1,0)} \circ *
\end{equation*}
are the adjoint operators, with respect to $(\cdot, \cdot)$, of $\overline{\partial}$ and $\nabla^{(1,0)}$, respectively.

\begin{defn} \textnormal{(\cite{B}, Section 6)} Let $(M,g)$ be an $n$-dimensional Riemannian manifold and let $E'$ be a complex vector bundle over $M$. We say that a differential operator $\theta$ of order $2$ on $E'$ is \emph{strongly elliptic} if, when expressed in terms of local trivializations of $E$ on coordinate neighborhoods, yields a differential operator on vector-valued functions of the form 
\begin{equation*}
\theta(s) = \sum_{\ell,m=1}^n g^{\ell m} \partial_\ell \partial_m s + \sum_{m=1}^n A^m \partial_m s + B s
\end{equation*}
where $(g^{ij})$ is the local representation of the induced Riemannian metric $g$ on $T^*X$ and $A^1$, $\dotsc$, $A^n$, $B$ are smooth matrix-valued functions.

Let $\mathcal{X} = (X,\mathcal{U})$ be a real orbifold and let $g$ be a Riemannian metric on $\mathcal{X}$. Let $\mathcal{E}'$ be a complex orbifold vector bundle over $\mathcal{X}$. A morphism of orbisheaves $\theta : \mathcal{A}^0(\mathcal{E}') \to \mathcal{A}^0(\mathcal{E}')$ is said to be \emph{strongly elliptic} if, for each orbifold chart $(U,G,\varphi) \in \mathcal{U}$, $\theta_U : \mathcal{A}^0(E'_U) \to \mathcal{A}^0(E'_U)$ is a strongly elliptic differential operator of order $2$.
\end{defn}

The \emph{Laplacians associated to $D'$ and $D''$} are defined by 
\begin{equation*}
\Delta' := {\nabla^{(1,0)}}^* \circ \nabla^{(1,0)} + \nabla^{(1,0)} \circ {\nabla^{(1,0)}}^* \hspace{10pt} \text{ and } \hspace{10pt} \Delta'' := \overline{\partial}^* \circ \overline{\partial} + \overline{\partial} \circ \overline{\partial}^*.
\end{equation*}
A global section $s \in \mathcal{A}^{p,q}(\mathcal{X},\mathcal{E})$ is called \emph{($\Delta''$-)harmonic} if $\Delta''(s) = 0$. We denote the space of all harmonic $(p,q)$-forms with values in $\mathcal{E}$ by $\mathcal{H}^{p,q}(\mathcal{X},\mathcal{E})$. The fact that $\Delta''$ is self-adjoint and that $2\Delta''$ is strongly elliptic imply the following result.

\begin{thm} \textnormal{(\cite{B}, Theorem D and Theorem F)} \label{thm: Hodge Decomposition}
There exists an orthogonal decomposition
\begin{equation*}
\mathcal{A}^{p,q}(\mathcal{X},\mathcal{E}) = \overline{\partial} (\mathcal{A}^{p,q-1}(\mathcal{X},\mathcal{E})) \oplus \mathcal{H}^{p,q}(\mathcal{X},\mathcal{E}) \oplus \overline{\partial}^* (\mathcal{A}^{p,q+1}(\mathcal{X},\mathcal{E}))
\end{equation*}
and $\mathcal{H}^{p,q}(\mathcal{X},\mathcal{E})$ is finite dimensional.
\end{thm}

Note that the Bochner-Kodaira-Nakano identity
\begin{equation} \label{eq: BKN}
\Delta'' = \Delta' + \left [ \sqrt{-1}F_H, \Lambda \right ],
\end{equation}
which holds on holomorphic vector bundles over K\"ahler manifolds, extends to holomorphic orbifold vector bundles over K\"ahler orbifolds. Then, \Cref{thm: Hodge Decomposition} and \ref{eq: BKN} imply the following version of H\"ormander's estimates.

\begin{cor} \label{cor: Hormander} \textnormal{(c.f. \cite{D}, Chapter 7 Lemma 7.2 and Chapter 8 Theorem 4.5)}
Let $\mathcal{L}$ be a holomorphic line bundle over $\mathcal{X}$ and let $h$ be a Hermitian metric on $\mathcal{L}$ with positive curvature. There exists a constant $C$ such that for all $k$ sufficiently large, if $t \in \mathcal{A}^{0,1}(\mathcal{X}, \mathcal{E} \otimes \mathcal{L}^k)$ satisfies $\overline{\partial} t = 0$ there exists $s \in \mathcal{A}^0(\mathcal{X}, \mathcal{E} \otimes \mathcal{L}^k)$ such that
$$
\|s\|^2\leq \frac{C}{k} \|t\|^2.
$$
\end{cor}

\section{Local Reproducing Kernels} \label{local reproducing kernels}

In this section, we will recall the notion of local reproducing kernels from \cite{BBS}. Then, following \cite{RT}, we will construct local reproducing kernels for orbifold vector bundles.
Finally, we will show that local reproducing kernels approximate the global Bergman kernels.

Let $U \subseteq \mathbb{C}^n$ be a connected open neighborhood of the origin. Let $H: U \to M_{r \times r}(\mathbb{C})$ be a smooth matrix-valued function on $U$ that is Hermitian and let $\phi: U \to \mathbb{R}$ be a smooth real-valued function on $U$ such that $\omega := \frac{\sqrt{-1}}{2\pi}\partial\overline{\partial}\phi$ is K\"ahler. For any $k \in \mathbb{N}$ a smooth $\mathbb{C}^r$-valued function $s \in C^\infty(U,\mathbb{C}^r)$ on $U$, set
\begin{equation*}
\|s\|_k = \int_U s(x)^t H(x) \overline{s(x)} e^{-k\phi(x)} \frac{\omega^n}{n!}(x)
\end{equation*}
(which could be infinite).  We define the subset $\mathcal{H}_{k}(U) \subseteq C^\infty(U,\mathbb{C}^{r})$ by
    \begin{equation*}
        \mathcal{H}_{k}(U) = \left \{ u \in H^0(U, \mathbb{C}^r) : \|u\|_{k} < \infty \right \}.
    \end{equation*}
Then, $( \cdot, \cdot )_k : \mathcal{H}_{k}(U) \times \mathcal{H}_{k}(U) \to \mathbb{C}$ defined by
\begin{equation*}
(s,t)_k = \int_U s(x)^t H(x) \overline{t(x)} e^{-k\phi(x)} \frac{\omega^n}{n!}(x)
\end{equation*} 
is an inner product on $\mathcal{H}_{k}(U)$. If $W$ is a complex vector space, we can extend the inner product $( \cdot, \cdot)_{k}$ to a pairing $(\cdot, \cdot)_{k} : \mathcal{H}_{k}(U) \times \mathcal{H}_{k}(U) \otimes W \to \overline{W}$ by using the formula
$$
(u, v \otimes w)_{k} = (u,v)_{k} \,\overline{w},
$$
where $\overline{w}$ is the image of $w$ under the conjugation map $W \to \overline{W}$. Fix a smooth cutoff function $\chi : U \to \mathbb{R}$ such that $\chi$ is compactly supported in $U$ and $\chi \equiv 1$ in $\frac{1}{2}U$.

\begin{defn}
Suppose that $\{\mathcal{K}_{k}\}_{k \gg 0}$ is a sequence in $C^\infty\left (U \times U, \mathbb{C}^r \boxtimes \overline{\mathbb{C}^r} \right )$. We say that $\mathcal{K}_{k}$ for $k \gg 0$ are \emph{local reproducing kernels modulo} $O(k^{-N-1})$ \emph{for} $\mathcal{H}_{k}(U)$ if, for all $x \in \frac{1}{2} U$ and for all $u \in \mathcal{H}_{k}(U)$,
    \begin{equation*}
        (\chi u, \mathcal{K}_{k}(\cdot,x))_{k} = u(x) + O\left ( e^{k\phi(x)/2} k^{-N-1} \right ) \|u\|_{k}. 
    \end{equation*}
\end{defn}

The construction of local reproducing kernels involves the following notion.

\begin{defn}
    Let $f : U \to \mathbb{R}$ be a smooth function. We say that $g : U \times U \to \mathbb{C}$ is an \emph{almost analytic extension} of $f$ if, for all $x \in U$, $g(x,\overline{x}) = f(x)$ and $D\overline{\partial} g(x,\overline{x}) = 0$ for any differential operator $D$ on $U \times U$.
    
    If $M = (m_{ab})_{1 \leq a,b \leq r}$ is a smooth $M_{r\times r}(\mathbb{C})$-valued function on $U$, then we say that a smooth $M_{r\times r}(\mathbb{C})$-valued function $\tilde{M} = (\tilde{m}_{ab})_{1 \leq a,b \leq r}$ on $U \times U$ is an \emph{almost analytic extension} of $M$ if the real and imaginary parts of $\tilde{m}_{ab}$ are almost analytic extension of the real and imaginary parts of $m_{ab}$, respectively, for all $0 \leq a,b \leq r$.

    If $f : U \to \mathbb{R}$ is a smooth function and $g : U \times U \to \mathbb{C}$ is an almost analytic extension, we say that $g$ is \emph{symmetric} if $g(x, y) = \overline{g(\overline{y}, \overline{x})}$ for all $x,y \in U$.
\end{defn}

It is shown in \cite{M} and Section 3.3.3.1 of \cite{R} that symmetric almost-analytic extensions exist after replacing $U$ with a smaller open polydisk centered at the origin, if necessary. Let $\psi(x,y)$ be a symmetric almost analytic extension of $\phi(x)$ and let $\overline{H}(x,y)$ be an almost analytic extension of $\overline{H}(x)$. 

\begin{lem} \label{lem: diastasis estimate} \textnormal{(\cite{BBS}, Equation 2.7) (\cite{R}, Lemma 3.8)}
After replacing $U$ with a smaller neighborhood if necessary, there exists a positive constant $\delta > 0$ such that
\begin{equation} \label{eq: diastasis estimate}
    -\phi(x) + \psi(y,\overline{x}) - \phi(y) + \psi(x,\overline{y}) \leq - \delta|x-y|^2
\end{equation}
for all $x,y \in U$.
\end{lem}

The following proposition from \cite{BBS} shows that local reproducing kernels exist. Denote by $\{e_1, \dotsc, e_r\}$ the standard basis vectors for $\mathbb{C}^r$. We identify local sections $t(y,x)$ of $\mathbb{C}^r \boxtimes \overline{\mathbb{C}^r}$ with $M_{r \times r}(\mathbb{C})$-valued functions by writing the $\overline{e_i} \otimes e_j$ component of $t(y,x)$ in the $i,j$-th entry of an $r \times r$ matrix. We write $f = O(k^{-\infty})$ to say that $f = O(k^{-m})$ for any $m$.

\begin{prop} \textnormal{(\cite{BBS}, Proposition 2.7)} \label{prop: local reproducing kernel}
    Let $N \geq 1$ be a fixed integer. After replacing $U$ with a smaller neighborhood of the origin if necessary, there exist smooth matrix-valued functions $\tilde{b}_{q}(x,z) : U \times U \to M_{r \times r}(\mathbb{C})$ such that $\mathcal{K}_{k} : U \times U \to M_{r \times r}(\mathbb{C})$ defined by
    \begin{equation*}
         \overline{\mathcal{K}_{k}(y,x)} = \left ( \tilde{b}_{0} (x,\overline{y}) k^n + \tilde{b}_{1}(x,\overline{y})  k^{n-1} + \dotsb + \tilde{b}_{N}(x,\overline{y}) k^{n-N} \right ) \overline{H}(x,\overline{y})^{-1} e^{k\psi(x,\overline{y})}
    \end{equation*}
    is a local reproducing kernel modulo $O(k^{-N-1})$ for $\mathcal{H}_{k}(U)$. Each $\tilde{b}_{q}$ can be written as a polynomial in the derivatives of $\psi$ and $H$. In particular, for all $x,z \in U$,
    \begin{equation*}
        \tilde{b}_{0}(x,z) = \Id_{r \times r} \hspace{10pt} \text{ and } \hspace{10pt} \tilde{b}_{1}(x,\overline{x}) = \frac{\Scal_\omega(x)}{2} \Id_{E} + \sqrt{-1} \Lambda F_{H}(x).
    \end{equation*}
    Moreover, if $D$ is any differential operator on $U \times U$, then 
    \begin{gather*}
    e^{-k(\phi(x)/2 + \phi(y)/2)} D\overline{\partial}_y \mathcal{K}_{k} \hspace{5pt} \text{ and } \hspace{5pt}  e^{-k(\phi(x)/2 + \phi(y)/2)}  D \partial_x \mathcal{K}_{k}    
    \end{gather*}
    are both $O(k^{-\infty})$.
\end{prop}

Now, let $\mathcal{X} = (X,\mathcal{U})$ be a compact complex orbifold with a K\"ahler form $\omega$. Let $\mathcal{L}$ be a holomorphic line bundle over $\mathcal{X}$ equipped with a Hermitian metric $h$ such that the curvature form of $h$ is $-2\pi\sqrt{-1} \omega$. Let $\mathcal{E}$ be a holomorphic vector bundle of rank $r$ over $\mathcal{X}$ equipped with a Hermitian metric $H$. For each $k \in \mathbb{N}$, set $\mathcal{E}(k) :=  \mathcal{E} \otimes \mathcal{L}^k$. Fix an orbifold chart $(U,G,\varphi)$ such that $\varphi(0) = x_0$, $G = G_{x_0}$, and $E_U$ and $L_U$ are trivial vector bundles. Note that, by \Cref{lem: orbifold chart}, we can find such an orbifold chart over any point $x_0 \in X$. Moreover, by \Cref{lem: orbifold vector bundle trivialization} we may assume that there exists a representation $\rho^k$ of $G$ such that 
\begin{equation*}
(g^{-1})^* (\tilde{x},v) = (g \cdot \tilde{x}, \rho^k(g) v)
\end{equation*}
for any $(\tilde{x},v) \in E_U \otimes L_U^k$. So, we can view orbifold sections of $\mathcal{E}(k)$ on $U$ as $G$-equivariant $\mathbb{C}^r$-valued functions on $U$. For the sake of convenience, we will write
\begin{equation*}
\sigma(\tilde{x}) = \sigma_U(\tilde{x})
\end{equation*}
for any local section $\sigma$ of an orbifold vector bundle over $\mathcal{X}$. 

Set $\phi := -\log(h)$. Fix a nonnegative integer $N$. For any $k \in \mathbb{N}$, define the subset $\mathcal{H}_{k,G}(U) \subseteq \mathcal{H}_k(U)$ by
\begin{equation*}
\mathcal{H}_{k,G}(U) = \left \{ u \in \mathcal{H}_k(U) :  \text{ $u$ is $G$-equivariant} \right \}.
\end{equation*} 
Then,
\begin{equation*}
( \cdot, \cdot )_{k,G} = \frac{1}{|G|} (\cdot, \cdot)_k
\end{equation*}
is an inner product on $\mathcal{H}_{k,G}(U)$ and we denote the induced norm by $\| \cdot \|_{k,G}$.
Fix a smooth cutoff function $\chi : U \to \mathbb{R}$ such that $\chi$ is compactly supported in $U$, $\chi \equiv 1$ in $\frac{1}{2}U$, and $\chi$ is $G$-invariant. By \Cref{prop: local reproducing kernel}, we can find local reproducing kernels $\mathcal{K}_{k}$ modulo $O(k^{-N-1})$ for $\mathcal{H}_{k}(U)$. We define $\mathcal{K}_{k}^{\av} : U \times U \to M_{r \times r}(\mathbb{C})$ by
\begin{equation*}
    \mathcal{K}_{k}^{\av} (\tilde{y},\tilde{x}) := \frac{1}{|G|} \sum_{g,h \in G} \overline{\rho^k(h^{-1})} \mathcal{K}_{k} (g \cdot \tilde{y}, h \cdot \tilde{x}) \rho^k(g^{-1})^t.
\end{equation*}

\begin{lem} \label{lem: local reproducing kernel}
$\mathcal{K}_{k}^{\av}$ is a local reproducing kernel mod $O(k^{-N-1})$ for the subspace $\mathcal{H}_{k,G}(U) \subseteq \mathcal{H}_{k}(U)$. In other words, for all $\tilde{x} \in 
\frac{1}{2} U$ and for all $u \in \mathcal{H}_{k,G}(U)$,
    \begin{equation*}
        (\chi u, \mathcal{K}_{k}^{\av}(\cdot,\tilde{x}))_{k,G} = u(\tilde{x}) + O\left ( e^{k\phi(\tilde{x})/2} k^{-N-1} \right ) \|u\|_{k,G}. 
    \end{equation*}
\end{lem}

\begin{proof}

Let $u \in \mathcal{H}(U)_{k,G}$ be given. Then, for any $g,h \in G$,
\begin{equation*}
\begin{split}
& \int_U \rho^{k}(h^{-1}) \overline{\mathcal{K}_{k} (g \cdot \tilde{y}, h \cdot \tilde{x})} \rho^{k}(g^{-1})^* \overline{H}(\tilde{y})e^{-k\phi(\tilde{y})} u(\tilde{y}) \chi(\tilde{y}) \frac{\omega^n}{n!} (\tilde{y})\\
= \hspace{2pt} & \rho^{k}(h^{-1})\int_U \overline{\mathcal{K}_{k} (g \cdot \tilde{y}, h \cdot \tilde{x})}  \overline{H}(g \cdot \tilde{y}) e^{-k\phi(g \cdot \tilde{y})} \rho^{k} (g) u(\tilde{y}) \chi(\tilde{y}) \frac{\omega^n}{n!}(\tilde{y}) \\
= \hspace{2pt} & \rho^{k}(h^{-1})\int_U \overline{\mathcal{K}_{k} (g \cdot \tilde{y}, h \cdot \tilde{x})}  \overline{H}(g \cdot \tilde{y}) e^{-k\phi(g \cdot \tilde{y})} u(g \cdot \tilde{y}) \chi(\tilde{y}) \frac{\omega^n}{n!} (\tilde{y}) \\
= \hspace{2pt} & \rho^{k}(h^{-1})\int_U \overline{\mathcal{K}_{k} (g \cdot \tilde{y}, h \cdot \tilde{x})}  \overline{H}(g \cdot \tilde{y}) e^{-k\phi(g \cdot \tilde{y})} u(g \cdot \tilde{y}) \chi(g \cdot \tilde{y}) \frac{(g^*\omega)^n}{n!} (\tilde{y})\\
= \hspace{2pt} & \rho^{k}(h^{-1})\int_U \overline{\mathcal{K}_{k} (\tilde{y}, h \cdot \tilde{x})}  \overline{H}(\tilde{y}) e^{-k\phi(\tilde{y})} u(\tilde{y}) \chi(\tilde{y}) \frac{\omega^n}{n!} (\tilde{y}) \\
= \hspace{2pt} & \rho^{k}(h^{-1}) u(h \cdot \tilde{x}) + \rho^{k}(h^{-1}) O \left ( e^{k\phi(h \cdot \tilde{x})/2} k^{-N-1} \right ) \|u\|_{k,G} \\
= \hspace{2pt} & u(\tilde{x}) + O\left ( e^{k\phi(\tilde{x})/2} k^{-N-1} \right ) \|u\|_{k,G}.
\end{split}
\end{equation*}
As a result, up to an error term of $O\left ( e^{k\phi(\tilde{x})/2} k^{-N-1} \right ) \|u\|_{k,G}$,
\begin{equation*}
     (\chi u, \mathcal{K}_{k}^{\av}(\cdot,\tilde{x}))_{k,G} = \frac{1}{|G|^2} \sum_{g,h \in G} u(\tilde{x}) = u(\tilde{x}).
\end{equation*}

\end{proof}

We will use \Cref{lem: local reproducing kernel} to show that the local reproducing kernels approximate the global Bergman Kernel up to an error term. To do this, we will need the following bound on smooth sections of holomorphic orbifold vector bundles.

\begin{prop} \textnormal{(c.f. \cite{RT}, Corollary 4.4)} \label{prop: standard estimate}
There exists a constant $C$ such that, for any section $s \in \mathcal{A}^0(\mathcal{X}, \mathcal{E}(k))$ and $x \in X$, we have
\begin{equation} \label{eq: standard estimate}
\|s(x)\|_{H \otimes h^k(x)} \leq C(k^{-\frac{1}{2}}\|\overline{\partial} s\|_{C^0(\mathcal{X})} + k^{\frac{n}{2}} \|s\|_{L^2(\mathcal{X})}).
\end{equation}
\end{prop}

Let $H^0(\mathcal{X},\mathcal{E} (k))$ be the set of global holomorphic sections of $\mathcal{E} (k)$. By \Cref{thm: Hodge Decomposition}, $H^0(\mathcal{X},\mathcal{E} (k))$ is finite dimensional.
Set $d_k = \dim(H^0 (\mathcal{X}, \mathcal{E}(k)))$. Choose an orthonormal basis $\{s_i\}_{i=1}^{d_k}$ of $H^0 (\mathcal{X}, \mathcal{E} (k))$. The Bergman kernel $K_k$ is the section of the vector bundle $\mathcal{E}(k) \boxtimes  \overline{\mathcal{E}(k)}$ given by
$$
K_k(x,y) = \sum_{i=1}^{d_k} s_i(x) \otimes \overline{s_i}(y)
$$
and satisfies the property that, for any element $s \in H^0(\mathcal{X},\mathcal{E}(k))$ and $x \in X$,
$$
s(x) = (s, K_k(\cdot,x)).
$$
The Bergman function $B_k$ is the section of $\mathcal{E}(k)^* \boxtimes \mathcal{E}(k)$ obtained from $\overline{K_k}$ by using the isomorphism $ \overline{\mathcal{E}(k)} \cong \mathcal{E}(k)^*$ induced by the metric $H\otimes h^k$.

\begin{prop} \label{prop: C^0 expansion}
There is a neighborhood $U_0 \subseteq \frac{1}{4}U$ of the origin such that
\begin{equation*}
K_k(x,y) = \mathcal{K}_k^{\av}(x,y) + O(k^{n-N-1})
\end{equation*}
on $\varphi(U_0) \times \varphi(U_0)$ with respect to the $C^0$ topology. Here, the error term $O(k^{n-N-1})$ is bounded with respect to the metric induced by $H\otimes h^k$.
\end{prop}

\begin{proof}

We will provide a sketch of the proof; full details can be found in Theorem 3.1 of \cite{BBS} and Proposition 4.6 of \cite{RT}. We write
\begin{equation*}
K_k(\tilde{x},\tilde{y}) = (K_k)_{U \times U}(\tilde{x},\tilde{y}).
\end{equation*}
Then,
\begin{equation*}
\begin{split}
\left ( \chi(\cdot) K_k(\cdot, \tilde{x}),  \mathcal{K}_k^{\av}(\cdot, \tilde{y}) \right )_{k,G} & =  K_k(\tilde{y},\tilde{x}) + \sum_{i=1}^{d_k}  O(e^{k\phi(\tilde{y})/2} k^{-N-1}) s_i(\tilde{x}) \|s_i\|_{k,G} \\
& = K_k(\tilde{y},\tilde{x}) + O(e^{k\phi(\tilde{y})/2}e^{k\phi(\tilde{x})/2} k^{-N-1} )\sum_{i=1}^{d_k} \|s_i(\tilde{x})\|_{H\otimes h^k(\tilde{x})}.
\end{split}
\end{equation*}
Since the evaluation map
\begin{equation*}
\ev_{\tilde{x}} (s) = s(\tilde{x})
\end{equation*}
has a kernel of dimension at least $d_k -r$, \Cref{prop: standard estimate} implies that
\begin{equation*}
\sum_{i=1}^{d_k} \|s_i(\tilde{x})\|_{H\otimes h^k(\tilde{x})} = O(k^{\frac{n}{2}}).
\end{equation*}
Therefore,
\begin{equation*}
\left ( \chi(\cdot) K_k(\cdot, x),  \mathcal{K}_k^{\av}(\cdot, y) \right ) =  K_k(y,x) + O(e^{k\phi(y)/2}e^{k\phi(x)/2} k^{\frac{n}{2} -N-1}).
\end{equation*}
This implies that, up to an error term of $O(e^{k\phi(y)/2}e^{k\phi(x)/2} k^{n-N-1})$,
\begin{equation*}
K_k(x,y) = \chi(x)\mathcal{K}_k^{\av}(x,y) - u_y(x),
\end{equation*}
where
\begin{equation*}
u_y(x) =  \chi(x)  \mathcal{K}_k^{\av}(x, y) - \left ( \chi(\cdot)  \mathcal{K}_k^{\av}(\cdot, y), K_k(\cdot, x) \right ). 
\end{equation*}
We observe that $u_y$ is precisely the orthogonal projection of $\chi(\cdot)  \mathcal{K}_k^{\av}(\cdot, y)$ onto $H^0(\mathcal{X},\mathcal{E}(k))^\perp$. 

Fix an open neighborhood $U_0$ of the origin such that $U_0 \subseteq \frac{1}{4}U$ and $U_0$ is $G$-invariant. By \Cref{lem: diastasis estimate}, there exists $\delta > 0$ such that
\begin{equation*}
 \left \| (\overline{\partial}\chi) \mathcal{K}_k^{\av}(\cdot,y) \right \|_{C^0(\mathcal{X})} = O(e^{-\delta k})
\end{equation*}
for all $y \in U_0$. Additionally, by
\Cref{prop: local reproducing kernel}, 
\begin{equation*}
\left \| \chi(\cdot) \overline{\partial} \mathcal{K}_k^{\av}(\cdot,y) \right \|_{C^0(\mathcal{X})} = O\left ( \frac{1}{k^\infty}  \right ).
\end{equation*}
As a result,
\begin{equation*}
\left \| \overline{\partial}\left (  \chi(\cdot)  \mathcal{K}_k^{\av}(\cdot,y) \right ) \right \|_{C^0(\mathcal{X})} = O\left ( \frac{1}{k^\infty}  \right ).
\end{equation*}
Then, by \Cref{cor: Hormander} and \Cref{prop: standard estimate},
\begin{equation*}
\|u_y(x)\| = O \left ( \frac{1}{k^\infty} \right ).
\end{equation*}

\end{proof}

In fact, the local reproducing kernels $\mathcal{K}_k^{\av}$ approximate the global Bergman kernels $K_k$ with respect to the $C^\infty$ topology. To show this, we will need the following lemma.

\begin{lem} \textnormal{(\cite{RT}, Lemma 4.9)} \label{lem: derivatives}
Let $U_0 \subseteq U$ be an open subset. Suppose that $\{f_k(\tilde{x},\tilde{y})\}$ is a sequence of functions on $U_0 \times U_0$ such that 
\begin{equation*}
(\overline{\partial}_{\tilde{x}} f_k)e^{-k(\phi(\tilde{x})+\phi(\tilde{y}))/2} \hspace{10pt} \text{ and } \hspace{10pt} (\partial_{\tilde{y}} f_k)e^{-k(\phi(\tilde{x})+\phi(\tilde{y}))/2} 
\end{equation*} 
are both $O(k^{-\infty})$. Suppose also that
\begin{equation*}
f_k e^{-k(\phi(\tilde{x})+\phi(\tilde{y}))/2} = O(k^m).
\end{equation*}
Then, for any differential operator $D$ in $\tilde{x}$ and $\tilde{y}$ of order $p$,
\begin{equation*}
(Df_k) e^{-k(\phi(\tilde{x})+\phi(\tilde{y})))/2} = O(k^{m+p})
\end{equation*}
on $\frac{1}{2} U_0 \times \frac{1}{2} U_0$.
\end{lem}

Let $p$ be a nonnegative integer. Applying \Cref{prop: C^0 expansion} and \Cref{lem: derivatives} to the difference
\begin{equation*}
K_k - \mathcal{K}_k^{\av},
\end{equation*}
we obtain the following result.

\begin{prop} \label{prop: C^p expansion}
There is a neighborhood $U_0 \subseteq U$ of the origin such that
\begin{equation*}
K_k(x,y) = \mathcal{K}_k^{\av}(x,y) + O(k^{n-N+p-1})
\end{equation*}
on $\varphi(U_0) \times \varphi(U_0)$ with respect to the $C^p$ topology.
\end{prop}

\section{Coefficients of the Local Reproducing Kernels} \label{coefficients}

In this section, we will analyze the coefficients in the expansion of the local reproducing kernels given in \Cref{prop: local reproducing kernel}. We will show that the coefficients can always be expressed as linear combination of certain matrix products. We will use this crucially in the next section where we prove \Cref{thm: main} .

Let $V$ be an $r$ dimensional complex vector space. Then, for any nonnegative integer $i \geq 0$, $\Sym^i V$ is an $r_i$ dimensional complex vector space, where $r_i = \binom{r+i-1}{r-1}$. We denote by $\mathfrak{s}^i : \End(V) \to \End(\Sym^i V)$ the Lie algebra homomorphism corresponding to $\Sym^i : \GL(V) \to \GL(\Sym^i V)$. Now, let $\{e_1, \dotsc, e_r\}$ be a basis for $V$. For any multi-index $\mu = (\mu_1, \dotsc, \mu_r)$ such that $\mu \geq 0$, we set
\begin{equation*} 
    e_\mu := \frac{{e_1}^{\mu_1} \dotsb {e_r}^{\mu_r}}{ \sqrt{\mu_1! \dotsb \mu_r!}} \in \Sym^{|\mu|}V.
\end{equation*}
For any nonnegative integer $i \geq 0$, we make the identification $\Sym^i V \cong \mathbb{C}^{r_i}$, by using the basis 
\begin{equation} \label{eq: symmetric basis}
    \left \{ e_\mu : \mu \in (\mathbb{Z}_{\geq 0})^r \text{ and } |\mu| = i  \right \}.
\end{equation}
Once a basis for $V$ has been chosen, we will use the above bases to view $\Sym^i$ and $\mathfrak{s}^i$ as maps from $M_{r \times r}(\mathbb{C})$ to $M_{r_i \times r_i}(\mathbb{C})$. It is a fact that, for any Hermitian matrix $H \in M_{r \times r}(\mathbb{C})$, $\Sym^i (H)$ is also a Hermitian matrix and that
\begin{equation*}
    \langle v_1 \dotsb v_i, w_1 \dotsb w_i \rangle_{\Sym^i (H)} = \sum_{\sigma \in \mathfrak{S}_i} \langle v_{1}, w_{\sigma(1)} \rangle_{H} \dotsb \langle v_{i}, w_{\sigma(i)} \rangle_{H} 
\end{equation*}
for all vectors $v_1$, $\dotsc$, $v_i$, $w_1$, $\dotsc$, $w_i$ in $\mathbb{C}^r$.

Let $U \subseteq \mathbb{C}^n$ be a connected neighborhood of the origin. Suppose that $H : U \to M_{r \times r}(\mathbb{C})$ is a Hermitian metric on $U \times \mathbb{C}^r$ and that the function $\phi : U \to \mathbb{R}$ defined by $\phi := -\log(\det(H))$ is smooth and $\omega := \frac{\sqrt{-1}}{2\pi} \partial \overline{\partial} \phi$ is a K\"ahler form on $U$. Set $H^i := \Sym^i (H)$. For any positive integers $i, k > 0$ and $s \in C^\infty(U,\mathbb{C}^{r_i})$ set,
\begin{equation*}
\|s\|_{i,k} = \int_U s(x)^t H^i(x) \overline{s(x)} e^{-k\phi(x)} \frac{\omega^n}{n!}(x).
\end{equation*}
As in \Cref{local reproducing kernels}, we define $\mathcal{H}_{i,k}(U)$ by
\begin{equation*}
\mathcal{H}_{i,k}(U) = \left \{ s \in H^0(U, \mathbb{C}^{r_i}) : \|s\|_{i,k} < \infty \right \}
\end{equation*}
and define an inner product $( \cdot, \cdot)_{i,k}$ on $\mathcal{H}_{i,k}(U)$ by 
\begin{equation*}
    (s,t)_{i,k} = \int_U s(x)^t H^i(x) \overline{t(x)} e^{-k\phi(x)} \frac{\omega^n}{n!}(x).
\end{equation*}

Let $N$ be a fixed nonnegative integer. By \Cref{prop: local reproducing kernel}, there exist smooth matrix-valued functions $\tilde{b}_{i,q}(x,z) : U \times U \to M_{r \times r}(\mathbb{C})$ such that $\mathcal{K}_{i,k} : U \times U \to M_{r_i \times r_i}(\mathbb{C})$ defined by
\begin{equation} \label{eq: local kernel}
\overline{\mathcal{K}_{k}(y,x)} = \left ( \tilde{b}_{i,0} (x,\overline{y}) k^n + \tilde{b}_{i,1}(x,\overline{y})  k^{n-1} + \dotsb + \tilde{b}_{i,N}(x,\overline{y}) k^{n-N} \right ) \overline{H}(x,\overline{y})^{-1} e^{k\psi(x,\overline{y})}
\end{equation}
is a local reproducing kernel modulo $O(k^{-N-1})$ for $\mathcal{H}_{i,k}(U)$.

We recall from \cite{BBS} the constructions of the functions $\tilde{b}_{i,q}$ in \eqref{eq: local kernel}. Define $\theta : U \times U \times U \to \mathbb{C}^{3n}$ by
\begin{equation*}
    \theta_j(x,y,z) =  \int_0^1\partial_j \psi(tx+(1-t)y,z) \,dt, 
\end{equation*}
where $\theta_j$ is the $j$-th coordinate of $\theta$ and $\partial_j \psi$ is the holomorphic derivative of $\psi(x,z)$ with respect to $x_j$. Because $\phi$ is strictly plurisubharmonic, the map
\begin{equation*}
    (x,y,z) \mapsto(x,y,\theta)
\end{equation*}
is a diffeomorphism near the origin in $\mathbb{C}^{3n}$. Define the function $\Delta_0$ by
\begin{equation*}
    \Delta_{0} (x,y,\theta) = \frac{\det (\partial_y\partial_z \psi(y,z))} {\det(\partial_z \theta(x,y,z))}
\end{equation*}
and set
\begin{equation*}
    \Delta_i (x,y,\theta) = \Delta_0(x,y,\theta)  \overline{H^i}(x,z)^{-1} \overline{H^i}(y,z).
\end{equation*}
Note that 
\begin{equation} \label{eq: base case}
\Delta_i(x,x,z(x,x,\theta)) = \det \partial_\theta (\theta) \Id = \Id.
\end{equation}

\begin{lem} \label{lem: recursive formula} \textnormal{(\cite{BBS}, Equation 2.13)}
    Let $U$ be a sufficiently small polydisk centered at the origin. Then, the coefficients $\tilde{b}_{i,q}(x,z)$ satisfy the recursive formula
    \begin{equation*}
        \sum_{l = 0}^m \frac{(D_\theta \cdot D_y)^l}{l!} \left ( \tilde{b}_{i,m-l}(x,z)  \Delta_i(x,y,\theta) \right)_{\mid y=x} = 0
    \end{equation*}
    for all $m > 0$.
\end{lem}

Fix a positive $N \geq 1$ and a nonnegative integer $i \geq 0$. Fix a neighborhood $U$ of the origin such that there exist smooth matrix-valued functions $\tilde{b}_{i,q} : U \times U \to M_{r_i \times r_i}(\mathbb{C})$, for $q = 0, \dotsc, N$, that satisfy the properties in \Cref{prop: local reproducing kernel}. By replacing $U$ with a smaller neighborhood of the origin if necessary, we may assume that there exists a smooth function $\Phi : U \times U \to M_{r \times r}(\mathbb{C})$ such that
\begin{equation*}
H(y,z) = e^{-\Phi(y,z)}.    
\end{equation*}

We denote the space of all smooth $M_{r_i \times r_i}(\mathbb{C})$-valued functions on $U \times U$ by $C^\infty(U \times U, M_{r_i \times r_i}(\mathbb{C}))$. For each $m \geq 0$, we define $R_m$ as the subset of $C^\infty(U \times U, M_{r_i \times r_i}(\mathbb{C}))$ consisting of functions $F$ such that $F$ can be written as a linear combination of elements of the form
\begin{equation*}
    f(x,z) \mathfrak{s}^i(M_1(x,z)) \dotsb \mathfrak{s}^i(M_{m'}(x,z))
\end{equation*}
where $m' \leq m$, $M_l(x,z)$ is a smooth $M_{r \times r}(\mathbb{C})$-valued function on $U \times U$ for all $1 \leq l \leq m'$, and $f(x,z) \in C^\infty(U \times U)$. Note that 
\begin{equation*}
    R_{m_1} \cdot R_{m_2} \subseteq R_{m_1+m_2}    
\end{equation*}
for all nonnegative integers $m_1 \geq 0$ and $m_2 \geq 0$. 

\begin{prop} \label{prop: shape of coefficients}
    For each $q = 0, \dotsc, N$, $\tilde{b}_{i,q}(x,z) \in R_q$.
\end{prop}

\begin{proof}

We will show first that 
\begin{equation} \tag{\textasteriskcentered} \label{eq: calc for prop}
    {\partial_z}^\beta{\partial_y}^\alpha \left ( \overline{H^i}(x,z)^{-1}  \overline{H^i}(y,z)\right )_{\mid y = x} \in R_{|\alpha|}
\end{equation}
for any multi-indices $\alpha$ and $\beta$. Consider the case when $|\alpha| = 1$. Without loss of generality, we may assume that ${\partial_y}^\alpha = \partial_{y_1}$. Then,
\begin{equation*}
    {\partial_z}^\beta \partial_{y_1} \left ( \overline{H^i}(x,z)^{-1} \overline{H^i}(y,z) \right ) = {\partial_z}^\beta \left ( \overline{H^i}(x,z)^{-1}\overline{H^i}(y,z) \mathfrak{s}^i\left ( \overline{M(y,z)} \right ) \right ),
\end{equation*}
where 
\begin{equation*}
    M(y,z) = \sum_{n=1}^\infty \frac{1}{n!}\ad (\Phi(y,z))^{n-1}  (\partial_{y_1}\Phi(y,z)).
\end{equation*}
Since
\begin{equation*}
    (D_z)^{\beta'} \left ( \overline{H^i}(x,z)^{-1} \overline{H^i}(y,z) \right)_{\mid y = x} = 0
\end{equation*}
for any multi-index $\beta'$,
\begin{equation*}
    {\partial_z}^\beta \partial_{y_1} \left ( \overline{H^i}(x,z)^{-1} \overline{H^i}(y,z) \right )_{\mid y = x} =   \mathfrak{s}^i\left ({\partial_z}^\beta \left ( \overline{M(x,z)} \right ) \right ) \in R_1.
\end{equation*} 
The proof of \eqref{eq: calc for prop} when $|\alpha| > 1$ is similar; for convenience of the reader we omit the proof.

Next, we view $z = z(x,y,\theta)$ as a function of $x$, $y$, and $\theta$ and show that
\begin{equation} \tag{\textasteriskcentered \textasteriskcentered} \label{eq: calc for prop 2}
{\partial_\theta}^\beta {\partial_y}^\alpha \left ( \overline{H^i}(x,z)^{-1}  \overline{H^i}(y,z)\right )_{\mid y = x} \in R_{|\alpha|}
\end{equation}
for any multi-indices $\alpha$ and $\beta$. By the Chain Rule, 
\begin{equation*}
\partial_{\theta_j} = \sum_{k=1}^n\deriv{z_k}{\theta_j} \partial_{z_k}
\end{equation*}
and
\begin{equation*}
    \partial_{y_j} = \partial_{y_j} + \sum_{k=1}^n\deriv{z_k}{y_j} \partial_{z_k}.
\end{equation*}
Then, \eqref{eq: calc for prop 2} follows from \eqref{eq: calc for prop}.

Finally, we prove the proposition by induction. By \Cref{lem: recursive formula}, 
\begin{equation*}
\tilde{b}_{i,0}(x,z) = \Id \in R_0.
\end{equation*}
Now, let $q_0$ be a positive integer and suppose that $\tilde{b}_{i,q} \in R_q$ for all $q \leq q_0 - 1$. Let $1 \leq l \leq q_0$ be given. By \eqref{eq: base case} and \Cref{lem: recursive formula}, it suffices to show that  
\begin{equation*}
    \frac{(D_\theta \cdot D_y)^l}{l!} \left ( \tilde{b}_{i,q_0-l} \Delta_i \right ) \in R_{q_0}
\end{equation*}
for all $l = 1, \dotsc, q_0$. Note that any derivative in $y$ or $\theta$ of an element of the form
\begin{equation*}
    f(x,z) \mathfrak{s}^i(M_1(x,z)) \dotsb \mathfrak{s}^i(M_p(x,z))
\end{equation*}
just replaces $f$ or one of $M_1$, $\dotsc$, $M_p$ with its derivative. So, if $\tilde{b}_{i,q} \in R_q$, then any derivative in $y$ or $\theta$ of $\tilde{b}_{i,q}$ is also an element of $R_q$. So, it will be enough to show that
\begin{equation} \tag{\textasteriskcentered \textasteriskcentered \textasteriskcentered} \label{eq: calc for prop 3}
{\partial_\theta}^\beta{\partial_y}^\alpha \Delta_{i}(x,y,\theta)_{\mid y = x} \in R_{|\alpha|}
\end{equation}
for any multi-indices $\alpha$ and $\beta$. We know that
\begin{equation*}
    \Delta_{i}(x,y,\theta) = \Delta_0(x,y,\theta) \overline{H^i}(x,z)^{-1} \overline{H^i}(y,z)
\end{equation*}
where $\Delta_0$ is a scalar function. As a result, \eqref{eq: calc for prop 3} follows from \eqref{eq: calc for prop 2}.

\end{proof}

From now on, we will restrict our attention to the case when the rank of $E$ is two. Let $V$ be a complex vector space of dimension $2$ and let $\{e_1,e_2\}$ be a basis for $V$. Let $i \geq 0$ be a nonnegative integer. Let $(A_{\mu\nu})_{|\mu|=|\nu| = i}$ denote the coordinates on $\GL(\Sym^i V)$ induced from the basis \eqref{eq: symmetric basis} for $\Sym^i V$. So, for any $A \in \GL(\Sym^i V)$,
$$
A(e_\nu) = \sum_{|\mu| = i} A_{\mu\nu}(A) e_\mu .
$$
Then, a direct computation of the derivative $\mathfrak{s}^i = d \Sym^i \mid_{\Id} : \End(V) \to \End(\Sym^i V)$ shows that
\begin{equation} \label{eq: s^i calculation}
\begin{split}
    \mathfrak{s}^i \left ( \begin{pmatrix} A_{11} & A_{12} \\ A_{21} & A_{22}\end{pmatrix}\right )  = &  \sum_{|\mu| = i} (\mu_1A_{11} + \mu_2 A_{22}) \tanv{A_{\mu\mu}} \\
    & + \sum_{|\mu| = i, \hspace{2pt} \mu_1 \geq 1} \sqrt{\mu_1(\mu_2+1)} A_{12}  \tanv{A_{(\mu_1,\mu_2)(\mu_1-1,\mu_2+1)}} \\
    & + \sum_{|\mu| = i, \hspace{2pt} \mu_2 \geq 1} \sqrt{(\mu_1+1)\mu_2} A_{21} \tanv{A_{(\mu_1 ,\mu_2)(\mu_1+1,\mu_2-1)}}.
\end{split} 
\end{equation}

\begin{prop} \label{prop: diff op}
Let $p \geq 1$ be a positive integer. For each $j = 1, \dotsc, p$, let $M_j$ be a two-by-two matrix. We write
\begin{equation*}
    M_j = 
    \begin{pmatrix}
        M_{j,11} & M_{j,12} \\
        M_{j,21} & M_{j,22} 
    \end{pmatrix}
\end{equation*}
for each $j = 1, \dotsc, p$. For any positive integer $i \geq 1$, denote by $A^i = (A^i_{\mu\nu})_{|\mu|=|\nu| = i}$ the matrix product
\begin{equation*}
    A^i = \mathfrak{s}^i \left ( M_1 \right ) \dotsb \mathfrak{s}^i \left ( M_p \right ).
\end{equation*}
Then, there exists a linear differential operator
\begin{equation*}
    L_0\left (T_1,T_2, \tanv{T_1}, \tanv{T_2}, M_{1,11}, \dotsb, M_{p,22} \right )
\end{equation*}
of order $p$ such that, for all $\mu = (\mu_1,\mu_2) \in (\mathbb{Z}_{\geq 0})^2$,
\begin{equation*}
   A^{\mu_1+\mu_2}_{\mu \mu} {T_1}^{\mu_1}{T_2}^{\mu_2} = L_0 \left ({T_1}^{\mu_1}{T_2}^{\mu_2} \right ).
\end{equation*}
Additionally, there exist linear differential operators
\begin{gather*}
    L_1\left (T_1,T_2, \tanv{T_1}, \tanv{T_2}, M_{1,11}, \dotsb, M_{p,22} \right ) \hspace{5pt} \text{ and } \hspace{5pt} L_{-1}\left (T_1,T_2, \tanv{T_1}, \tanv{T_2}, M_{1,11}, \dotsb, M_{p,22} \right )
\end{gather*}
of order $p+1$ such that
\begin{equation*}
    \sqrt{\mu_1(\mu_2 + 1)}A^{\mu_1 + \mu_2}_{(\mu_1,\mu_2)(\mu_1-1,\mu_2+1)} {T_1}^{\mu_1-1} T_2^{\mu_2} = L_1 \left ( {T_1}^{\mu_1} T_2^{\mu_2} \right ),
\end{equation*}
whenever $\mu_1 > 0$ and
\begin{equation*}
    \sqrt{(\mu_1+1)\mu_2 }A^{\mu_1 + \mu_2}_{(\mu_1,\mu_2)(\mu_1+1,\mu_2-1)} {T_1}^{\mu_1} T_2^{\mu_2-1} = L_{-1} \left ( {T_1}^{\mu_1} T_2^{\mu_2} \right ),
\end{equation*}
whenever $\mu_2 > 0$. The coefficients of $L_0$, $L_1$, and $L_{-1}$ are polynomials in $T_1$, $\frac{1}{T_1}$, $T_2$, $\frac{1}{T_2}$, $M_{1,11}$, $\dotsc$, $M_{p,22}$.
\end{prop}

\begin{proof}

Let $i \geq 1$ be a positive integer. For each $j = 1, \dotsc, p$, we will write 
\begin{equation*}
\mathfrak{s}^i(M_j) = \left ( M^i_{j, \mu \nu} \right )_{|\mu| = |\nu| = i}
\end{equation*}
for the entries of the matrix $\mathfrak{s}^i(M_j)$. Then,
\begin{equation*}
    A^i_{\mu\nu} = \sum_{|\mu^1| = \dotsb = |\mu^{p-1}| = i} M^i_{1,\mu\mu^1} M^i_{2,\mu^1\mu^2} \dotsb M^i_{p,\mu^{p-1}\nu}.
\end{equation*}
Denote by $\Gamma(p,\mu,\nu)$ the set of all sequences $\{\mu^j\}_{j=0}^p$ of multi-indices in $\left ( \mathbb{Z}_{\geq 0} \right )^2$ of length $p+1$ such that $|\mu^j| = i$ for all $j = 0, \dotsc, p$, $\mu^0 = \mu$, $\mu^p = \nu$, and $|\mu^j - \mu^{j-1}| \leq 2$ for all $j = 1, \dotsc, p$. We infer from \eqref{eq: s^i calculation} that
\begin{equation} \tag{\textasteriskcentered} \label{eq: product formula}
    A^i_{\mu\nu} = \sum_{\{\mu^j\}_{j=0}^p \in \Gamma(p,\mu,\nu)} M^i_{1,\mu^0\mu^1} M^i_{2,\mu^1\mu^2} \dotsb M^i_{p,\mu^{p-1}\mu^p}.
\end{equation}

Fix $\mu = (\mu_1,\mu_2) \in \left ( \mathbb{Z}_{\geq 0} \right )^2$ with $|\mu| = i$. Define $F : \{ (\mu_1',\mu_2') \in \mathbb{Z} \times \mathbb{Z} : \mu_1' + \mu_2' = i\} \to \mathbb{Z}$ by
\begin{equation*}
    F(\mu_1',\mu_2') = \mu_2' - \mu_2.
\end{equation*}
Denote by $\Gamma'(p,\ell,\ell')$ the set of all integer sequences $\{\ell_j\}_{j=0}^p$ of length $p+1$ such that $\ell_0 = \ell$, $\ell_p = \ell'$, and $|\ell_j - \ell_{j-1}| \leq 1$ for all $j = 1, \dotsc, p$. Then, $F$ induces a bijection between $\Gamma(p,\mu,(\mu_1-\ell, \mu_2 + \ell))$ and $\Gamma'(p,0,\ell)$. Define $P : M_{2 \times 2} (\mathbb{C}) \times \{(a,b) \in \mathbb{Z}^2 : |a-b| \leq 1 \} \to \mathbb{C}$ by
\begin{equation*}
    P \left (\begin{pmatrix} M_{11} & M_{12} \\ M_{21} & M_{22} \\ \end{pmatrix}, a,b \right) =  
    \begin{cases}
        (\mu_1 - a)  M_{11} + (\mu_2 + a) M_{22} & \text{ if  $a = b$} \\
        \sqrt{ \max\{(\mu_1 - a)(\mu_2 + a +1),0\} } M_{12}  & \text{ if $a+1 = b$, and} \\
        \sqrt{ \max\{(\mu_1 - a + 1) (\mu_2 + a),0\} } M_{21} & \text{ if $a - 1 = b$.} \\
    \end{cases}
\end{equation*}
Then, \eqref{eq: product formula} can be rewritten as
\begin{equation} \tag{\textasteriskcentered \textasteriskcentered} \label{eq: product formula 1}
    A^i_{(\mu_1,\mu_2) (\mu_1 - \ell, \mu_2 + \ell)} = \sum_{\{\ell_j\}_{j=0}^p \in \Gamma'(p,0,\ell)} \prod_{j = 1}^{p} P(M_j, \ell_{j-1},\ell_{j}).
\end{equation}

We will first consider the case when $\ell = 0$. Let $\gamma = \{\ell_j\}_{j=0}^p \in \Gamma' (p,0,0)$ be given. We observe that, whenever $P(M_j,a,a+1)$ for some integer $a \in \mathbb{Z}$ appears in the product $\prod_{j = 1}^{p} P(M_j, \ell_{j-1},\ell_{j})$, $P(M_{j'},a+1,a)$ must also appear in the product, because the sequence $\{\ell_j\}_{j=0}^p$ starts and ends at $0$. Let $p'$ denote the number of integers $j$ such that $\ell_j = \ell_{j+1}$. After rearranging the terms in the product and relabeling the sequence $\{M_j\}_{j=1}^p$, we can choose a sequence of integers $\{a_j\}_{j=1}^{(p+p')/2}$ so that we can rewrite the product $\prod_{j = 1}^{p} P(M_j, \ell_{j-1},\ell_{j})$ as
\begin{equation*}
    \left ( \prod_{j=1}^{p'} (\mu_1 - a_j)M_{j,11} + (\mu_2 + a_j)M_{j,22}\right ) \left ( \prod_{j = 1}^{(p-p')/2}  \left ( \mu_1 - a_{p'+j} \right )\left ( \mu_2 + a_{p'+j} + 1\right ) M_{p'+2j-1,12}M_{p'+2j,21}  \right ). 
\end{equation*}
Define a differential operator $L_\gamma\left (T_1,T_2, \tanv{T_1}, \tanv{T_2}, M_{1,11}, \dotsb, M_{p,22} \right )$ by
\begin{equation*}
    L_\gamma = L_{\gamma,1} \circ \dotsb \circ L_{\gamma, (p + p')/2}
\end{equation*}
where
\begin{equation*}
    L_{\gamma,j} (f)=
    \begin{cases}
        M_{j,11}{T_1}^{a_j + 1}\tanv{T_1}\left ( {T_1}^{-a_j}  f \right ) + M_{j,22}{T_2}^{-a_j + 1}\tanv{T_2}\left ( {T_2}^{a_j}  f \right ) & \text{ if $1 \leq j \leq p'$, and } \\
        M_{2j-p'-1,12}M_{2j-p',21}{T_1}^{a_j + 1}{T_2}^{-a_j} \tanv{T_1}\tanv{T_2} \left ( {T_1}^{-a_j}{T_2}^{a_j+1} f \right ) & \text{ if $p'+1 \leq j \leq (p+p')/2$.}
    \end{cases}
\end{equation*}
Note that $L_\gamma$ does not depend on $\mu^0$ and that $L_\gamma({T_1}^{\mu_1}{T_2}^{\mu_2}) = 0$ whenever $a_j \geq \mu_1$ or $a_j \leq -\mu_2-1$ for some $j$. Setting 
\begin{equation*}
L_0 := \sum_{\gamma \in \Gamma'(p,0,0)} L_\gamma   
\end{equation*}
and using \eqref{eq: product formula 1}, we see that
\begin{equation*} 
    A^i_{\mu \mu} {T_1}^{\mu_1}{T_2}^{\mu_2} = L_0 \left ({T_1}^{\mu_1}{T_2}^{\mu_2} \right ).
\end{equation*}
By construction, $L_0$ is a differential operator of order $p$, $L_0$ does not depend on $\mu$ or $i$, and the coefficients of $L_0$ are polynomials in $T_1$, $\frac{1}{T_1}$, $T_2$, $\frac{1}{T_2}$, $M_{1,11}$, $\dotsc$, $M_{p,22}$. 

Next, we will consider the case when $\ell = 1$. Let $\gamma = \{\ell_j\}_{j=0}^p \in \Gamma'(p,0,1)$ be given. Let $j_0$ be the largest number such that $\ell_{j_0} = 0$. Then, it is necessary that $\ell_{j_0 + 1} = 1$. So $\gamma' = \{\ell_j\}_{j=0}^{j_0} \in \Gamma'(j_0,0,0)$ and $\gamma'' = \{\ell_j\}_{j=j_0+1}^p \in \Gamma'(p-j_0-1,1,1)$. By what we have already shown, we can rearrange the sequence $\{M_j\}_{j=1}^p$ to find differential operators $L_{\gamma'}$ and $L_{\gamma''}$ such that
\begin{equation*}
\begin{split}
    & \prod_{j = 1}^{p} P(M_j, \ell_{j-1},\ell_{j}) \\
    = \hspace{2pt} & \left ( {T_1}^{-\mu_1}{T_2}^{-\mu_2} L_{\gamma'} \left ({T_1}^{\mu_1}{T_2}^{\mu_2} \right ) \right )\sqrt{\mu_1(\mu_2+1)}M_{j_0+1,12} \left ( {T_1}^{-\mu_1+1}{T_2}^{-\mu_2-1} L_{\gamma''} \left ({T_1}^{\mu_1-1}{T_2}^{\mu_2+1} \right )\right ).
\end{split}
\end{equation*}
Define a differential operator $L_\gamma\left (T_1,T_2, \tanv{T_1}, \tanv{T_2}, M_{1,11}, \dotsb, M_{p,22} \right )$ by
\begin{equation*}
    L_\gamma(f) =  M_{j_0+1,12}  \tanv{T_2} \bigg ( L_{\gamma''}  \bigg ( T_2\tanv{T_1} \bigg ( L_{\gamma'}\bigg ( f \bigg ) \bigg ) \bigg ) \bigg ).
\end{equation*}
Set
\begin{equation*}
    L_1 := \sum_{\gamma \in \Gamma'(p,0,1)} L_\gamma.
\end{equation*}
Then,
\begin{equation*}
    \sqrt{\mu_1(\mu_2 + 1)}A^i_{(\mu_1,\mu_2)(\mu_1-1,\mu_2+1)} {T_1}^{\mu_1-1} T_2^{\mu_2} = L_1 \left ( {T_1}^{\mu_1} T_2^{\mu_2} \right ).
\end{equation*}
As before, $L_1$ is a differential operator of order $p+1$, $L_1$ does not depend on $\mu$ or $i$, and the coefficients of $L_1$ are polynomials in $T_1$, $\frac{1}{T_1}$, $T_2$, $\frac{1}{T_2}$, $M_{1,11}$, $\dotsc$, $M_{p,22}$.

Finally, the case when $\ell = -1$ is analogous to the case when $\ell = 1$.

\end{proof}

\section{Local Expansion} \label{expansion}

In this section, we will prove \Cref{thm: main}. Let $\mathcal{X}$ be a compact complex orbifold with abelian singularities equipped with a K\"ahler form $\omega$. Suppose that, for each $x \in X$, the stabilizer group $G_x$ is abelian and has rank at most two. Let $\mathcal{E}$ be an orbi-ample vector bundle of rank $2$ over $\mathcal{X}$ and set $\mathcal{L} := \det \mathcal{E}$. Let $H$ be a Hermitian metric on $\mathcal{E}$ such that the curvature form of $h := \det H$ is $-2\pi\sqrt{-1}\omega$, and fix $x_0 \in X$. Using \Cref{lem: orbifold chart} and \Cref{lem: orbifold vector bundle trivialization}, we fix an orbifold chart $(U,G,\varphi)$ such that $\varphi(0) = x_0$, $G = G_{x_0}$ acts linearly on $U$,  $E := E_U$ is a trivial vector bundle, and there exists a representation $\rho$ of $G$ such that 
\begin{equation*}
(g^{-1})^* (\tilde{x},v) = (g \cdot \tilde{x}, \rho(g) v)
\end{equation*}
for any $(\tilde{x},v) \in E$.

By our assumption $G = \mathbb{Z}/a\mathbb{Z} \times \mathbb{Z}/b\mathbb{Z}$ where $a, b \geq 1$, and that $G$ acts linearly on $U$. Since $G$ is abelian, we can choose a different trivialization of $E$ if necessary so that
\begin{equation*}
    \rho (g_1,g_2) = 
    \begin{pmatrix}
        \zeta_a(g_1) \zeta_b(g_2) & 0 \\
        0 & \zeta_a'(g_1) \zeta_b'(g_2)
    \end{pmatrix}
\end{equation*}
for any $(g_1,g_2) \in G$, where $\zeta_a(g_1)$ and $\zeta_a'(g_1)$ are $a$-th roots of unity and $\zeta_b(g_2)$ and $\zeta_b'(g_2)$ are $b$-th roots of unity. We set $\phi := -\log(\det H)$ so that $h = e^{-\phi}$. Additionally, for any nonnegative integer $i \geq 0$ and positive integer $k \geq 1$, we set $E^{i,k} = \Sym^i E \otimes L^k$ and $\rho^{i,k} := \Sym^i \rho \otimes (\det \rho)^{\otimes k}$.

Now, let $\{e_1,e_2\}$ be a basis for $V$. Note that $\End(\Sym^i V \otimes \det(V)^{\otimes k}) \cong \End(\Sym^i V)$. With respect to the basis \eqref{eq: symmetric basis}, we can view $\tau_{i,k}$ as a map from $M_{(i+1) \times (i+1)}(\mathbb{C})$ to $M_{2 \times 2}(\mathbb{C})$. Then, \eqref{eq: dual} and \eqref{eq: s^i calculation} imply that
\begin{equation} \label{eq: tauik calculation}
\begin{split}
    & \tau_{i,k}( (A_{\mu\nu})_{|\mu|=|\nu| = i}) \\
    & = 
    \begin{pmatrix}
        k \sum_{|\mu|=i} A_{\mu\mu} + \sum_{|\mu| = i} \mu_1 A_{\mu\mu} & \sum_{|\mu|=i, \hspace{2pt} \mu_2 \geq 1} \sqrt{(\mu_1+1)\mu_2} A_{(\mu_1+1,\mu_2-1)\mu} \\
        \sum_{|\mu|=i, \hspace{2pt} \mu_1 \geq 1} \sqrt{\mu_1(\mu_2+1)} A_{(\mu_1-1,\mu_2+1)\mu} & k \sum_{|\mu|=i} A_{\mu\mu} + \sum_{|\mu| = i} \mu_2 A_{\mu\mu}
    \end{pmatrix} \\
\end{split}
\end{equation}

The following lemma is a reformulation of the claim inside the proof of Theorem 3.3 in \cite{RT}. We provide the proof here to fix a small mistake found in \cite{RT}.

\begin{lem} \label{lem: claim}
    Let $(U,G,\varphi) \in \mathcal{U}$ be an orbifold chart such that $U$ is a neighborhood of the origin, $G$ acts linearly on $U$, and \eqref{eq: diastasis estimate} holds on $U$. Let $g,h \in G$ with $g \neq h$. Define $\eta : U \to \mathbb{C}$ by
    \begin{equation*}
        \eta(\tilde{x}) = e^{\psi(h \cdot \tilde{x}, \overline{g \cdot \tilde{x}}) - \phi(g \cdot \tilde{x})}.
    \end{equation*}
    Then, for any nonnegative integer $s \geq 0$ and any $x \in U$,
    \begin{equation*}
        \|h \cdot \tilde{x} - g \cdot \tilde{x}\|^{2s} \eta(\tilde{x})^k = O(k^{-s}).
    \end{equation*}
\end{lem}

\begin{proof}
    If $\|h \cdot \tilde{x} - g \cdot \tilde{x}\| = 0$, then the desired estimate holds. So, suppose that $\|h \cdot \tilde{x} - g \cdot \tilde{x}\|^2 \neq 0$. By \eqref{eq: diastasis estimate} and the fact that $\phi$ is $G$-invariant and Hermitian,
    \begin{equation*}
    \begin{split}
        2 \Real (\psi(h \cdot \tilde{x}, \overline{g \cdot \tilde{x}}) - \phi(g \cdot \tilde{x})) & =  - \phi(h \cdot \tilde{x}) + \psi(g \cdot \tilde{x}, \overline{h \cdot \tilde{x}}) - \phi(g \cdot \tilde{x}) + \psi(h \cdot \tilde{x}, \overline{g \cdot \tilde{x}})  \\
        & \leq -\delta \|h \cdot \tilde{x} - g \cdot \tilde{x}\|^2.
    \end{split}
    \end{equation*}
    Recall that if $z$ is a complex number with $\Real(z) < 0$, then $|e^{kz}| \leq \frac{s!}{(-k\Real(z))^s}$ for all $k \geq 0$. It follows that
    \begin{equation*}
        \eta(\tilde{x})^k \leq \left (\frac{2}{\delta} \right )^s \frac{s!}{k^s \|h \cdot \tilde{x} - g \cdot \tilde{x}\|^{2s}}.
    \end{equation*}
\end{proof}

Fix nonnegative integers $N, p \geq 0$. Let $\zeta_{\ord(\mathcal{X})}$ be a primitive $\ord(\mathcal{X})$-th root of unity. Let $\{c_{ij}\}_{i,j \geq 0}$ be a set of nonnegative constants such that the polynomial
\begin{equation*}
\sum_{i,j \geq 0} c_{ij} \left ( \sum_{\mu \in (\mathbb{Z}_{\geq 0})^2, |\mu|=i} {T_1}^{\mu_1} {T_2}^{\mu_2} \right ){T_3}^j
\end{equation*} 
has total order at least $5N+5p+5$ at $(T_1,T_2,T_3) = ({\zeta_{\ord(\mathcal{X})}}^{m_1}, {\zeta_{\ord(\mathcal{X})}}^{m_2}, {\zeta_{\ord(\mathcal{X})}}^{m_1+m_2})$ for all $m_1$,$m_2$ $\in \mathbb{Z}$ such that $({\zeta_{\ord(\mathcal{X})}}^{m_1},{\zeta_{\ord(\mathcal{X})}}^{m_2})$ does not equal $(1,1)$. For all nonnegative integers $k \geq 0$, we can write \eqref{eq: orbifold Bergman function} as
\begin{equation*}
    B_k^{\orb}(x) = \sum_{i,j \geq 0} c_{ij}\tau_{i,k+j} \left ( K_{i,k+j}(x, x) \overline{H^i}(x)e^{-k\phi(x)}  \right ).
\end{equation*}

\begin{thm} \label{thm: main local}
There exists a neighborhood of the origin $U_0 \subseteq U$ and local sections $b_0, \dotsc, b_N$ of $ \End(\mathcal{E})$ on $U$ such that
    \begin{equation*}
        B_k^{\orb}(\tilde{x}) =  b_0(\tilde{x}) k^{n+1} + \dotsb + b_{N}(\tilde{x}) k^{n-N+1} + O(k^{n - N})
    \end{equation*}
    on $U_0$ with respect to the $C^p$-norm. The matrix-valued functions $b_q$ depend only on the constants $\{c_{ij}\}_{i,j \geq 0}$ and the derivatives of the metrics involved. In particular,
    \begin{equation*}
        b_0(x) = \sum_{i,j \geq 0} c_{ij} (i+1) \Id_E
    \end{equation*}
    and
    \begin{equation*}
    b_1(x) =  \sum_{i,j \geq 0} c_{ij} \left (nj + \frac{(i+1)(2i+1)}{2} \Scal_\omega(x) \right )\Id_E.
    \end{equation*}
\end{thm}

\begin{proof}

Define $\mathcal{B}^{\orb}_k : U \times U \to M_{2 \times 2}(\mathbb{C})$ by
\begin{equation*}
\mathcal{B}^{\orb}_k(\tilde{x}) = \sum_{i,j \geq 0} c_{ij}\tau_{i,k+j} \left (  \overline{\mathcal{K}_{i,k+j}^{\av} (\tilde{x}, \tilde{x})} \overline{H^i}(\tilde{x}) e^{-(k+j)\phi(\tilde{x})}  \right ),
\end{equation*}
where
\begin{equation*}
\overline{\mathcal{K}_{i,k+j}^{\av} (\tilde{x}, \tilde{x})} = \frac{1}{|G|} \sum_{g,h \in G} \rho^{i,k+j}(h^{-1})\left ( \sum_{q = 0}^{N+p} \tilde{b}_{i,q} (h \cdot \tilde{x},\overline{g \cdot \tilde{x}}) k^{n-q} \right ) \overline{H^i}(h \cdot \tilde{x},\overline{g \cdot \tilde{x}})^{-1} e^{k\psi(h \cdot \tilde{x},\overline{g \cdot \tilde{x}})} \rho^{i,k+j}(g^{-1})^*.
\end{equation*}
By \Cref{prop: C^0 expansion}, there exists an open subset $U_0 \subseteq U$ 
such that
\begin{equation*}
    B_k^{\orb}(\tilde{x}) = \mathcal{B}_k^{\orb}(\tilde{x}) + O(k^{n-N}) 
\end{equation*}
with respect to the $C^p$-topology.

We observe that
\begin{equation*}
\begin{split}
    B_k^{\orb}(\tilde{x}) & = \sum_{i,j \geq 0} c_{ij}\tau_{i,k+j} \left ( \frac{1}{|G|} \sum_{g,h \in G} \rho^{i,k+j}(h^{-1}) \overline{\mathcal{K}_{i,k+j} (g \cdot \tilde{x}, h \cdot \tilde{x})} \rho^{i,k+j}(g^{-1})^* \overline{H^i}(\tilde{x})e^{-(k+j)\phi(\tilde{x})}  \right ) \\
    & = S_1 + S_2,
\end{split}
\end{equation*}
where $S_1$ consists of terms with $g=h$ and $S_2$ consists of terms with $g \neq h$. We will show below that $S_2 = O(k^{n-N})$. On the other hand,
\begin{equation*}
    \begin{split}
        S_1 
        & = \sum_{i,j \geq 0} c_{ij}\tau_{i,k+j} \left ( \frac{1}{|G|} \sum_{g \in G} \rho^{i,k+j}(g^{-1}) \overline{\mathcal{K}_{i,k+j} (g \cdot \tilde{x}, g \cdot \tilde{x})} \overline{H^i}(g \cdot \tilde{x})e^{-(k+j)\phi(g \cdot \tilde{x})} \rho^{i,k+j}(g)  \right ) \\
        & = \sum_{i,j \geq 0} c_{ij}\tau_{i,k+j} \left ( \frac{1}{|G|} \sum_{g \in G} \rho^{i,k+j}(g^{-1}) \left ( \sum_{q =0}^{N+p} \tilde{b}_{i,q} (g \cdot \tilde{x}, \overline{g \cdot \tilde{x}}) (k+j)^{n-q} \right ) \rho^{i,k+j}(g) \right ).
    \end{split}
\end{equation*}
Thus, up to an error term of $O(k^{n-N})$ in the $C^p$-topology, we have that
\begin{equation} \tag{\textasteriskcentered} \label{eq: S_1}
\begin{split}
    S_1 & =   \sum_{q=0}^{N-1}  \sum_{i,j \geq 0} c_{ij}\frac{1}{|G|} \sum_{g \in G}\tau_{i,0} \left ( \rho^{i,0}(g^{-1}) \tilde{b}_{i,q}(g \cdot \tilde{x},\overline{g \cdot \tilde{x}}) \rho^{i,0}(g) \right ) (k+j)^{n-q} \\
    & +   \sum_{q=0}^{N-1} \sum_{i,j \geq 0} c_{ij} \frac{1}{|G|} \sum_{g \in G} \tr \left ( \rho^{i,0}(g^{-1})\tilde{b}_{i,q+1}(g \cdot \tilde{x},\overline{g \cdot \tilde{x}})\rho^{i,0}(g) \right ) \Id  (k+j)^{n-q},
\end{split}
\end{equation}
where $\tilde{b}_{i,-1} = 0$. By \eqref{eq: dual}, we can write
\begin{equation*}
\sum_{g \in G}\tau_{i,0} \left ( \rho^{i,0}(g^{-1}) \tilde{b}_{i,q}(g \cdot \tilde{x},\overline{g \cdot \tilde{x}}) \rho^{i,0}(g) \right ) = \sum_{g \in G} \rho(g^{-1}) \tau_{i,0} \left (  \tilde{b}_{i,q}(g \cdot \tilde{x},\overline{g \cdot \tilde{x}}) \right ) \rho(g).
\end{equation*}
So, $\sum_{g \in G} \rho(g^{-1}) \tau_{i,0} \left (  \tilde{b}_{i,q}(g \cdot \tilde{x},\overline{g \cdot \tilde{x}}) \right ) \rho(g)$ is a $G$-equivariant section of $\End(E)$ for each $i$ and $q$. Moreover, the function
\begin{equation*}
\frac{1}{|G|} \sum_{g \in G} \tr \left ( \rho^{i,0}(g^{-1})\tilde{b}_{i,q+1}(g \cdot \tilde{x},\overline{g \cdot \tilde{x}})\rho^{i,0}(g) \right )
\end{equation*}
is $G$-invariant, because $\tr$ is invariant under conjugation.

If $q < n$, we can expand $(k+j)^{n-q}$ using the binomial expansion 
\begin{equation*}
    (k+j)^{n-q} = \sum_{\ell = 0}^{N-1} \binom{n-q}{\ell} j^\ell k^{n-q-\ell}  + O(k^{n-N}).    
\end{equation*}
If $q > n$, then
\begin{equation*}
    \begin{split}
        (k+j)^{n-q} & = k^{n-N+1} \frac{k^{N-n-1}}{(k+j)^{q-n}} \\
        & = k^{n-N+1} \left (f(k) + \frac{r(k)}{(k+j)^{q-n}} \right ),
    \end{split}
\end{equation*}
where $f(k)$ and $r(k)$ are polynomials in $k$ such that $\deg (f(k)) = N-q-1$ and $\deg(r(k)) < q-n$. We see that
\begin{equation*}
     (k+j)^{n-q} = k^{n-N+1}f(k) + O(k^{n-N}).
\end{equation*}
So, we expand every power of $(k+j)$ in \eqref{eq: S_1} to obtain local sections $b_0, \dotsc, b_N$ of $\End(\mathcal{E})$ such that the sections do not depend on $k$ and
\begin{equation*}
    S_1 = b_0(x)k^{n+1} + \dotsb + b_N(x) k^{n-N+1} + O(k^{n-N}).
\end{equation*}
In particular,
\begin{equation*}
    b_0(x) = \sum_{i,j \geq 0} c_{ij} (i+1) \Id_E
\end{equation*}
and
\begin{equation*}
\begin{split}
    b_1(x) & = \sum_{i,j \geq 0} c_{ij} \left (nj + \tr\left ( \sqrt{-1} \Lambda F_{H^i} \right ) + \frac{i+1}{2} \Scal_\omega(x)   \right )\Id_E\\
    & =  \sum_{i,j \geq 0} c_{ij} \left (nj  + \frac{(i+1)(2i+1)}{2} \Scal_\omega(x) \right )\Id_E.
\end{split}
\end{equation*}

It remains to show that $S_2 = O(k^{n-N})$ in the $C^p$-topology. We write $S_2 = \frac{1}{|G|}\sum_{g \neq h} S_{g,h}$, where
\begin{equation*}
    S_{g,h} = \sum_{i,j \geq 0} c_{ij}\tau_{i,k+j} \left (  \rho^{i,k+j}(h^{-1}) \overline{\mathcal{K}_{i,k+j} (g \cdot \tilde{x}, h \cdot \tilde{x})}  \overline{H^i}(g \cdot \tilde{x}) e^{-(k+j)\phi(g \cdot \tilde{x})}\rho^{i,k+j}(g)  \right ).
\end{equation*}
Fix $g,h \in G$ such that $g \neq h$.
We replace $U_0$ with a smaller neighborhood of the origin, if necessary, such that \eqref{eq: diastasis estimate} holds, $M: U_0 \to M_{2 \times 2}(\mathbb{C})$ defined by 
\begin{equation*}
    M(x) =  \log \left ( \overline{H}(h \cdot \tilde{x}, \overline{g \cdot \tilde{x}} )^{-1} \overline{H} (g \cdot \tilde{x}) \right )
\end{equation*}
is well-defined, and $M(\tilde{x}) = O(\|h \cdot \tilde{x} - g \cdot \tilde{x}\|)$. Set
\begin{equation*}
    \eta(\tilde{x}) = e^{\psi(h\cdot \tilde{x}, \overline{g \cdot \tilde{x}}) - \phi(g \cdot \tilde{x})}.
\end{equation*}
Then,
\begin{equation*}
    S_{g,h} = \sum_{i,j \geq 0}   \sum_{q=0}^{N+p} c_{ij}  \tau_{i,k+j} \left (\rho^{i,k+j}(h^{-1}) \tilde{b}_{i,q}(h \cdot \tilde{x}, \overline {g \cdot \tilde{x}}) e^{\mathfrak{s}^i(M(\tilde{x}))} \rho^{i,k+j}(g) \right ) \eta(\tilde{x})^{k+j} (k+j)^{n-q} .
\end{equation*}
As before, we expand every power of $(k+j)$ and use the identity 
\begin{equation*}
\tau_{i,k+j}(A) = (k+j)\tr(A) \Id + \tau_{i,0}(A)
\end{equation*}
to write
\begin{equation*}
    S_{g,h} =   f_0(\tilde{x}) k^{n+1} + \dotsb  + f_{N+p+1}(\tilde{x}) k^{n-N-p}
\end{equation*}
where $f_0(\tilde{x})$, $\dotsc$, $f_{N+p+1}(\tilde{x})$ are functions with values in $M_{2 \times 2}(\mathbb{C})$. Let $\alpha$ be a multi-index in $(\mathbb{Z}_{\geq 0})^{2n}$ with $|\alpha| \leq p$. We will show that $D^\alpha f_q(\tilde{x}) = O(k^{q-N-1})$ for all $0 \leq q \leq N+p+1$.

We will first consider the case when $q \leq n$. Set
\begin{equation*}
\sigma_1 = \zeta_a({h_1}^{-1}g_1) \zeta_b({h_2}^{-1}g_2) \hspace{10pt} \text{  and } \hspace{10pt} \sigma_2 = \zeta_a'({h_1}^{-1}g_1) \zeta_b'({h_2}^{-1}g_2).    
\end{equation*}
Then,
\begin{equation*}
\begin{split}
    D^\alpha f_q(\tilde{x}) & = \sum_{\ell = 0}^q \binom{n-\ell}{q - \ell}  \sum_{i,j \geq 0}  c_{ij} \tau_{i,0} \left (\rho^{i,0}(h^{-1}) \beta_{i,\ell}(h \cdot \tilde{x}, \overline{g \cdot \tilde{x}}) e^{\mathfrak{s}^i(M(x))} \rho^{i,0}(g) \right )  j^{q - \ell}(\sigma_1\sigma_2)^{k+j} \eta(x)^{k+j} \\
    & \hspace{10pt} +  \sum_{\ell = 0}^{q+1} \binom{n+1-\ell}{q+1-\ell}  \sum_{i,j \geq 0}  c_{ij}  \tr\left (\rho^{i,0}(h^{-1}) \beta_{i,\ell}(h \cdot \tilde{x}, \overline{g \cdot \tilde{x}}) e^{\mathfrak{s}^i(M(\tilde{x}))} \rho^{i,0}(g) \right ) j^{q+1-\ell}(\sigma_1\sigma_2)^{k+j}\eta(\tilde{x})^{k+j},
\end{split}
\end{equation*}
where $\beta_{i,\ell}$ is a linear combination of elements of the form $(k+j)^m \gamma$
with $0 \leq m \leq p$ and $\gamma \in R_{\ell + p - m}$. We observe that, for any $0 \leq \ell \leq q$,
\begin{equation*}
\begin{split}
    & \tau_{i,0} \left (\rho^{i,0}(h^{-1}) \beta_{i,\ell}(h \cdot \tilde{x}, \overline{g \cdot \tilde{x}}) \left ( \sum_{n=2(N+p+1)}^\infty \frac{\mathfrak{s}^i(M(\tilde{x}))^n}{n!} \right )  \rho^{i,0}(g) \right ) \eta(\tilde{x})^{k} \\
    = \hspace{3pt} & O \left (k^p \|h \cdot \tilde{x} - g \cdot \tilde{x} \|^{2(N+p+1)} \right ) \eta(\tilde{x})^k \\
    & \text{ since $M(\tilde{x}) = O( \|h \cdot \tilde{x} - g \cdot \tilde{x}\|)$} \\
    = \hspace{3pt} & O(k^{-N-1}) \\
    & \text{ by \Cref{lem: claim}.}
\end{split}
\end{equation*}
Similarly,
\begin{equation*}
    \tr\left (\rho^{i,0}(h^{-1}) \beta_{i,\ell}(h \cdot \tilde{x}, \overline{g \cdot \tilde{x}}) \left ( \sum_{n=2(N+p+1)}^\infty \frac{\mathfrak{s}^i(M(\tilde{x}))^n}{n!} \right ) \rho^{i,0}(g) \right ) \eta(\tilde{x})^{k} = O(k^{-N-1}).
\end{equation*}
Thus, it suffices to show that, for any $0 \leq \ell \leq q$, $0 \leq m \leq p$, and $\gamma \in R_{\ell+p-m}$,
\begin{equation*}
    \sum_{i,j \geq 0}  c_{ij} \tau_{i,0} \left (\rho^{i,0}(h^{-1}) \gamma(h \cdot \tilde{x}, \overline{g \cdot \tilde{x}}) \left ( \sum_{n=0}^{2N+2p + 1} \frac{\mathfrak{s}^i(M(\tilde{x}))^n}{n!}\right ) \rho^{i,0}(g) \right )  j^{q - \ell}(\sigma_1\sigma_2)^{k+j} \eta(\tilde{x})^{k+j} = O(k^{q-N-m-1})
\end{equation*}
and
\begin{equation*}
    \sum_{i,j \geq 0}  c_{ij} \tr \left (\rho^{i,0}(h^{-1}) \gamma(h \cdot \tilde{x}, \overline{g \cdot \tilde{x}}) \left ( \sum_{n=0}^{2N+2p +1} \frac{\mathfrak{s}^i(M(\tilde{x}))^n}{n!}\right ) \rho^{i,0}(g) \right )  j^{q+1-\ell}(\sigma_1\sigma_2)^{k+j} \eta(\tilde{x})^{k+j} = O(k^{q-N-m-1}).
\end{equation*}

Let $0 \leq \ell \leq N$, $0 \leq m \leq p$ be given, and $\gamma \in R_{\ell + p - m}$. Set
\begin{equation*}
    A^i(\tilde{x}) = \gamma (h \cdot \tilde{x}, \overline{g \cdot \tilde{x}}) \left ( \sum_{n=0}^{2N+2p+1} \frac{\mathfrak{s}^i(M(\tilde{x}))^n}{n!} \right ).
\end{equation*}
By \eqref{eq: tauik calculation},
\begin{equation*}
    \tau_{i,0} \left ( \rho^{i,0}(h^{-1}) A^i(\tilde{x}) \rho^{i,0}(g) \right ) =
    \begin{pmatrix}
        P_1 & P_2 \\
        P_3 & P_4 \\
    \end{pmatrix}
\end{equation*}
where
\begin{equation*}
\begin{split}
    P_1 & = \sum_{|\mu| = i} \mu_1  A^i_{\mu\mu}(\tilde{x}) {\sigma_1}^{\mu_1}{\sigma_2}^{\mu_2},  \\
    P_2 & = \zeta_a({h_1}^{-1})\zeta_b(h_2^{-1}) \zeta'_a(g_1)\zeta'_b(g_2) \sum_{|\mu|=i, \hspace{2pt} \mu_2 \geq 1} \sqrt{(\mu_1+1)\mu_2} A^i_{(\mu_1+1,\mu_2-1)\mu}(\tilde{x}) {\sigma_1}^{\mu_1} {\sigma_2}^{\mu_2-1},\\
    P_3 & = \zeta_a(g_1)\zeta_b(g_2)\zeta_a'({h_1}^{-1}) \zeta'_b({h_2}^{-1}) \sum_{|\mu|=i, \hspace{2pt} \mu_1 \geq 1} \sqrt{\mu_1(\mu_2+1)} A^i_{(\mu_1-1,\mu_2+1)\mu}(\tilde{x}) {\sigma_1}^{\mu_1-1}{\sigma_2}^{\mu_2}, \text{ and }  \\
    P_4 & =  \sum_{|\mu| = i} \mu_2 A^i_{\mu\mu}(\tilde{x}) {\sigma_1}^{\mu_1} {\sigma_2}^{\mu_2}.
\end{split}
\end{equation*}
By \Cref{prop: diff op}, there exist differential operators $L_0$, $L_1$, and $L_{-1}$ such that
\begin{equation*}
\begin{split}
    P_1 & = \left (T_1\tanv{T_1} \right ) \circ L_0   \left ( \sum_{|\mu|=i} {T_1}^{\mu_1} {T_2}^{\mu_2} \right )  \Bigg |_{T_1 = \sigma_1, T_2 = \sigma_2}, \\
    P_2 & = \zeta_a({h_1}^{-1})\zeta_b(h_2^{-1}) \zeta'_a(g_1)\zeta'_b(g_2) L_1  \left ( \sum_{|\mu|=i} {T_1}^{\mu_1} {T_2}^{\mu_2} \right )  \Bigg |_{T_1 = \sigma_1, T_2 = \sigma_2}, \\
    P_3 & = \zeta_a(g_1)\zeta_b(g_2)\zeta_a'({h_1}^{-1}) \zeta'_b({h_2}^{-1}) L_{-1}  \left ( \sum_{|\mu|=i} {T_1}^{\mu_1} {T_2}^{\mu_2} \right )  \Bigg |_{T_1 = \sigma_1, T_2 = \sigma_2}, \text{ and } \\
    P_4 & =   \left ( T_2\tanv{T_2} \right ) \circ L_0   \left ( \sum_{|\mu|=i} {T_1}^{\mu_1} {T_2}^{\mu_2} \right )  \Bigg |_{T_1 = \sigma_1, T_2 = \sigma_2}.
\end{split}
\end{equation*}
Define differential operators $D_1$, $D_2$, $D_3$, and $D_4$ by
\begin{equation*}
\begin{split}
    D_1 & = \left ( T_3 \tanv{T_3} \right )^{q-\ell} \circ \left ( T_1\tanv{T_1} \right )\circ L_0,    \\
    D_2 & = \zeta_a({h_1}^{-1})\zeta_b(h_2^{-1}) \zeta'_a(g_1)\zeta'_b(g_2)   \left ( T_3 \tanv{T_3} \right )^{q - \ell} \circ L_{1}, \\
    D_3 & = \zeta_a(g_1)\zeta_b(g_2)\zeta_a'({h_1}^{-1}) \zeta'_b({h_2}^{-1})   \left ( T_3 \tanv{T_3} \right )^{q - \ell} \circ L_{-1}, \text{ and } \\
    D_4 & = \left ( T_3 \tanv{T_3} \right )^{q - \ell} \circ \left ( T_2\tanv{T_2} \right )\circ L_0.
\end{split}
\end{equation*}
Note that the orders of the differential operators $D_1$, $\dotsc$, $D_4$ are at most $q + 2N+ 3p + 1$ and that $(\sigma_1,\sigma_2) \neq (1,1)$ because $\rho$ is a faithful representation. So,
\begin{equation*}
\begin{split}
    & \sum_{i,j \geq 0} c_{ij} \tau_{i,0} \left ( \rho^{i,0}(h) A^i(\tilde{x}) \rho^{i,0}(g^{-1}) \right ) j^{q - \ell}(\sigma_1\sigma_2)^{j} \eta(\tilde{x})^{j} \\
    = \hspace{3pt} & 
    \begin{pmatrix}
        D_1 & D_2 \\
        D_3 & D_4\\
    \end{pmatrix}
    \left ( \sum_{i,j} c_{ij} \sum_{|\mu| = i} {T_1}^{\mu_1}{T_2}^{\mu_2}{T_3}^j \right )\Bigg |_{T_1 = \sigma_1, T_2 = \sigma_2, T_3 = \sigma_1\sigma_2 \eta(\tilde{x})}
\end{split}
\end{equation*}
has a zero of order at least $2N+2p+4$ at $\eta(\tilde{x}) = 1$. In particular,
\begin{equation*}
    \frac{\sum_{i,j \geq 0} c_{ij} \tau_{i,0} \left ( \rho^{i,0}(h) A^i(\tilde{x}) \rho^{i,0}(g^{-1}) \right ) j^{q - \ell}(\sigma_1\sigma_2)^{j} \eta(\tilde{x})^{j}}{(\eta(\tilde{x}) - 1)^{2N+2p+4}} = O(1).
\end{equation*}
Note that $\eta(\tilde{x})-1 = O(\|h \cdot \tilde{x} - g \cdot \tilde{x}\|)$. Thus, by \Cref{lem: claim},
\begin{equation*}
    \begin{split}
        \sum_{i,j \geq 0}  c_{ij} \tau_{i,0} \left (\rho^{i,0}(h) A^i(\tilde{x}) \rho^{i,0}(g^{-1}) \right )  j^{q - \ell}(\sigma_1\sigma_2)^{k+j} \eta(\tilde{x})^{k+j} &  = O(\|h \cdot \tilde{x} - g \cdot \tilde{x}\|^{2N+2p+4}) \eta(\tilde{x})^k \\
        & =  O(k^{-N-p-2}).
    \end{split}
\end{equation*}
An analogous argument implies that
\begin{equation*}
    \begin{split}
        \sum_{i,j \geq 0}  c_{ij} \tr \left (\rho^{i,0}(h) A^i(\tilde{x}) \rho^{i,0}(g^{-1}) \right )  j^{q - \ell}(\sigma_1\sigma_2)^{k+j} \eta(\tilde{x})^{k+j}  =   O(k^{-N-p-1}).
    \end{split}
\end{equation*}

It remains to consider the case when $q \geq n+1$. In the equation
\begin{equation*}
     (k+j)^{n-l} = k^{n-N+ 1}f(k) + O(k^{n-N}),
\end{equation*}
the coefficients of $f(k)$ are polynomials in $j$ of degree at most $N-l-1$. Thus, we can apply arguments similar to the ones above.

\end{proof}

\Cref{thm: main} now follows from \Cref{thm: main local}, the fact that $X$ is compact, and the fact that the local sections $b_0$, $\dotsc$, $b_N$ of $\End(\mathcal{E})$ do not depend on $k$.

\newpage

\appendix

\section{Orbifold Charts and Orbifold Vector Bundle Trivializations}

In this section, we record some lemmas that are helpful for local calculations. \Cref{lem: orbifold chart} says that orbifolds are locally quotients of open subset of $\mathbb{C}^n$ by linear automorphisms (taken from $\mathsection 1.1$ of \cite{S}, and we include the proof for the sake of completeness). \Cref{lem: orbifold vector bundle trivialization} says that if $\mathcal{E}$ is a holomorphic orbifold vector bundle over a complex orbifold $\mathcal{X} = (X,\mathcal{U})$, then $\mathcal{E}$ is locally isomorphic to $(U \times \mathbb{C}^r) / G$, where $U$ is an open subset of $\mathbb{C}^n$ and $G$ is a finite group that acts on $U$ and $\mathbb{C}^r$.

\begin{lem}  \textnormal{(\cite{S}, $\mathsection 1.1$)} \label{lem: orbifold chart}
Let $\mathcal{X} = (X,\mathcal{U})$ be a complex orbifold and let $x \in X$. Then, there exists an orbifold chart $(U, G, \varphi) \in \mathcal{U}$ such that $G = G_x$, $0 \in U$, $\varphi(0)= x$, and $G$ acts linearly on $U$.
\end{lem}

\begin{proof}

Choose a holomorphic orbifold chart $(U,G,\varphi) \in \mathcal{U}$ such that $x \in \varphi(U)$. Choose $\tilde{x} \in U$ such that $\varphi(\tilde{x}) = x$. Because $G$ is finite, we can find an arbitrarily small open $G$-stable neighborhood $S$ of $\tilde{x}$ in $U$ such that $G_x = G_{\tilde{x}} = G_S$. Moreover, by Lemma 2.10 of \cite{MM}, $G_x$ acts faithfully on $U$. Then, the inclusion map $\lambda : (S,G_x, \varphi_{\mid S}) \to (U,G,\varphi)$ is a holomorphic embedding between orbifold charts and $(S,G_x,\varphi_{\mid S}) \in \mathcal{U}$. So, after replacing $(U,G,\varphi)$ with $(S,G_x,\varphi_{\mid S})$, we may assume that $G = G_x$. We may further assume that there exists a coordinate system $z = (z_1, \dotsc, z_n)$ on $U$ such that $z(\tilde{x}) = 0$.

For each $i \in \{1, \dotsc, n\}$, define $w_i : U \to \mathbb{C}$ by
\begin{equation*}
w_i = \frac{1}{|G|} \sum_{g \in G}\sum_{j=1}^n \deriv{(z_i \circ g^{-1})}{z_j} (\tilde{x})z_j \circ g.
\end{equation*}
Then, $w = (w_1, \dotsc, w_n)$ is a coordinate system on a neighborhood of $\tilde{x}$ since
\begin{equation*}
\begin{split}
\deriv{w_i}{z_k} (\tilde{x}) & = \frac{1}{|G|} \sum_{g \in G} \sum_{j=1}^n \deriv{(z_i \circ g^{-1})}{z_j} (\tilde{x}) \deriv{(z_j \circ g)}{z_k} (\tilde{x}) \\  
    & = \delta_{ik}.
\end{split}
\end{equation*}
Furthermore, for any $h \in G$, we see that
\begin{equation*}
\begin{split}
w_i \circ h & = \frac{1}{|G|} \sum_{g \in G} \sum_{j=1}^n \deriv{(z_i \circ g^{-1})}{z_j} (\tilde{x})z_j \circ g \circ h \\
    & = \frac{1}{|G|} \sum_{g \in G} \sum_{j=1}^n \deriv{(z_i \circ h \circ g^{-1})}{z_j} (\tilde{x})z_j \circ g \\
    & = \frac{1}{|G|} \sum_{g \in G} \sum_{k=1}^n \sum_{j=1}^n \deriv{(z_i \circ h)}{z_k}(\tilde{x}) \deriv{(z_k \circ g^{-1})}{z_j} (\tilde{x})z_j \circ g \\
    & = \sum_{k=1}^n \deriv{(z_i \circ h)}{z_k}(\tilde{x}) w_k.
\end{split}
\end{equation*}
Define $A : G \to \GL_n(\mathbb{C})$ by
\begin{equation*}
    A(h) = \left ( \deriv{(z_i \circ h)}{z_j}(\tilde{x}) \right )_{1 \leq i,j \leq n}.
\end{equation*}
By the Chain Rule and the fact that $G$ fixes $\tilde{x}$, $A$ is a group homomorphism. Moreover, $A$ is injective by Lemma 2.10 of \cite{MM}. By our computations from above, we can write
\begin{equation*} \tag{\textasteriskcentered} \label{eq: coord}
    w \circ h = A(h) \circ w.
\end{equation*}

We can find a sufficiently small connected neighborhood $U' \subseteq U$ of $\tilde{x} \in U$ such that $w$ is an injective immersion on $U'$ and $G_{U'} = G$. Set $V = w(U')$ and define $\lambda : V \to U'$ by $\lambda = (w_{\mid U'})^{-1}$. Inverting the equation \eqref{eq: coord}, we see that
\begin{equation*}
    \lambda \circ A(h^{-1}) = h^{-1} \circ \lambda
\end{equation*}
for any $h \in G$. Then, $\lambda : (V,G, \varphi \circ \lambda) \to (U,G,\varphi)$ is a holomorphic embedding between orbifold charts.

\end{proof}

\begin{lem} 
\label{lem: orbifold vector bundle trivialization}
Let $\mathcal{E}$ be a holomorphic vector bundle over a complex orbifold $\mathcal{X} = (X,\mathcal{U})$. Let $x_0 \in X$. Then, there exists an orbifold chart $(U, G, \varphi) \in \mathcal{U}$, a trivialization $E_U \cong U \times \mathbb{C}^r$, a linear action of $G$ on $U$ and an action of $G$ on $\mathbb{C}^r$ such that $\varphi(0) = x_0$ and
\begin{equation*}
\rho_U(g^{-1})( (\tilde{x},v) ) = (g \cdot \tilde{x}, g \cdot v)
\end{equation*}
for all $g \in G$, $\tilde{x} \in U$ and $v \in \mathbb{C}^r$.
\end{lem}

\begin{proof}

By \Cref{lem: orbifold chart}, there exists an orbifold chart $(U,G,\varphi) \in \mathcal{U}$ such that $G = G_{x_0}$, $\varphi(0) = x_0$, and $G$ acts linearly on $U$. By replacing $U$ with a smaller neighborhood of the origin, if necessary, we may assume that the holomorphic vector bundle $E_U$ is the trivial bundle. Then, for any $g \in G$, there exists a matrix valued function $\tilde{\rho}_g : U \to \GL(r,\mathbb{C})$ such that
\begin{equation*}
\rho_U(g^{-1}) (\tilde{x},v) = (g \cdot \tilde{x}, \tilde{\rho}_g(\tilde{x}) v)
\end{equation*}
for all $g \in G$ and $(\tilde{x},v) \in E_U \cong U \times \mathbb{C}^r$. Note that $(0,0)$ is a fixed point for the action of $G$ on $E_U$.

We observe that 
\begin{equation*}
\frac{1}{|G|} \sum_{g \in G} \sum_{j=1}^n \tilde{\rho}_{g^{-1}} (0) \tilde{\rho}_g(\tilde{x})
\end{equation*}
is an isomorphism for all $\tilde{x}$ sufficiently near $0$. After replacing $U$ with a smaller neighborhood of $0$ if necessary, we define a trivialization of $E_U$ by
\begin{equation*}
(\tilde{x}, u(\tilde{x},v))  =  \left ( \tilde{x}, \quad  \frac{1}{|G|} \sum_{g \in G} \sum_{j=1}^n \tilde{\rho}_{g^{-1}} (0) \tilde{\rho}_g(\tilde{x}) v  \right ).
\end{equation*}
For any $h \in G$, $\tilde{x} \in U$, and $v \in \mathbb{C}^r$,  we observe that
\begin{equation*}
\begin{split}
\left(h \cdot \tilde{x},  u(h \cdot \tilde{x}, \tilde{\rho}_h( \tilde{x}) v) \right) & = \left (h \cdot \tilde{x},\quad \frac{1}{|G|} \sum_{g \in G} \sum_{j=1}^n \tilde{\rho}_{g^{-1}} (0) \tilde{\rho}_g(h \cdot \tilde{x}) \tilde{\rho}_h( \tilde{x}) v  \right ) \\
    & = \left( h \cdot \tilde{x}, \quad \frac{1}{|G|} \sum_{g \in G} \sum_{j=1}^n \tilde{\rho}_{g^{-1}} (0) \tilde{\rho}_{gh}(\tilde{x}) v  \right) \\
    & = \left ( h \cdot \tilde{x}, \quad \frac{1}{|G|} \sum_{g \in G} \sum_{j=1}^n \tilde{\rho}_{hg^{-1}} (0) \tilde{\rho}_{g}( \tilde{x}) v  \right ) \\
    & =  \left (h \cdot  \tilde{x}, \quad \tilde{\rho}_{h} (0)  \left ( \frac{1}{|G|} \sum_{g \in G} \sum_{j=1}^n \tilde{\rho}_{g^{-1}} (0) \tilde{\rho}_{g}(\tilde{x}) v \right )  \right ) \\
    & =  ( h \cdot  \tilde{x},  \tilde{\rho}_h(0) u(\tilde{x}, v)  ).
\end{split}
\end{equation*}
In other words,
\begin{equation*}
\rho_U(h^{-1})( \tilde{x}, u( \tilde{x}, v)) = (h \cdot \tilde{x},  \tilde{\rho}_h(0) u( \tilde{x}, v)).
\end{equation*}

\end{proof}

\section{Existence of the Weights} \label{Schur}

Let $d$ be a positive integer and let $\zeta$ be a primitive $d$-th root of unity. Let $p$ be a positive integer. In view of \Cref{thm: main}, we seek a set $\{c_{ij}\}_{i,j \geq 0}$ of nonnegative constants, only finitely many of which are nonzero, such that the polynomial
\begin{equation} \label{eq: Schur}
\sum_{i,j \geq 0} c_{ij} \left ( \sum_{\mu \in (\mathbb{Z}_{\geq 0})^2, |\mu|=i} {T_1}^{\mu_1} {T_2}^{\mu_2} \right ){T_3}^j
\end{equation} 
has total order $p$ at $(T_1,T_2,T_3) = (\zeta^{m_1}, \zeta^{m_2}, \zeta^{m_1+m_2})$ for all $m_1$,$m_2$ $\in \mathbb{Z}$ such that $(\zeta^{m_1},\zeta^{m_2})$ does not equal $(1,1)$.

\begin{lem} \label{lem: Schur in 2 vars}
    Let $S_{(\lambda_1,\lambda_2)}(z,w)$ denote the Schur polynomial associated to the partition $(\lambda_1, \lambda_2)$ in two variables $z$ and $w$. Then,
    \begin{equation*}
        S_{(\lambda_1,\lambda_2)}(z,w) = z^{\lambda_1}w^{\lambda_2} + \dotsb + z^{\lambda_2}w^{\lambda_1}.
    \end{equation*}
\end{lem}
\begin{proof}

We write $T = \frac{w}{z}$ so that, for any $\lambda \geq 0$,
\begin{equation*}
    \frac{1}{z^\lambda}H_\lambda(z,w) = 1 + T + \dotsb + T^\lambda.
\end{equation*}
By the Jacobi-Trudy identity,
\begin{equation*}
    S_{(\lambda_1,\lambda_2)}(z,w) = H_{\lambda_1}(z,w)H_{\lambda_2}(z,w) - H_{\lambda_1+1}(z,w)H_{\lambda_2-1}(z,w).
\end{equation*}
So, we see that
\begin{equation*}
    \begin{split}
        & \frac{1}{z^{\lambda_1 + \lambda_2}} S_{(\lambda_1,\lambda_2)}(z,w) \\
        & =(1 + T + \dotsb + T^{\lambda_1})(1 + T + \dotsb + T^{\lambda_2}) - (1 + T + \dotsb + T^{\lambda_1 + 1}) (1 + T + \dotsb + T^{\lambda_2 - 1}) \\
        & = \left ( \sum_{i=0}^{\lambda_2} T^i + \dotsb + T^{\lambda_1 + i} \right ) - \left ( \sum_{j=0}^{\lambda_2 - 1} T^j + \dotsb + T^{\lambda_1 + 1 + j} \right ) \\
        & = T^{\lambda_2} + \dotsb + T^{\lambda_1 + \lambda_2} - \sum_{j=0}^{\lambda_2 -1} T^{\lambda_1 + 1 + j} \\
        & = T^{\lambda_2} + \dotsb + T^{\lambda_1}.
    \end{split}
\end{equation*}
    
\end{proof}

The polynomial 
\begin{equation*}
q(z,w) = (1 + z + \dotsb + z^d)^p (1 + w + \dotsb + w^d)^p
\end{equation*}
has total order $p$ at $(z,w) = (\zeta^{m_1}, \zeta^{m_2})$ for all $m_1$,$m_2$ $\in \mathbb{Z}$ such that $(\zeta^{m_1},\zeta^{m_2})$ does not equal $(1,1)$. By \Cref{lem: Schur in 2 vars},
\begin{equation*}
S_{(\lambda_1,\lambda_2)}(z,w) = \left ( \sum_{\mu \in (\mathbb{Z}_{\geq 0})^2, |\mu|= \lambda_1 - \lambda_2 } z^{\mu_1} w^{\mu_2} \right ) (zw)^{\lambda_2},
\end{equation*}
for any partition $(\lambda_1, \lambda_2)$. As a result, $q(z,w)$ is Schur positive if and only if $q(z,w) = Q(z,w,zw)$, where $Q(T_1,T_2,T_3)$ is a polynomial of the form \eqref{eq: Schur} (to see this note that since $q(z,w)$ is a symmetric polynomial in two variables, it is sufficient to only consider Schur polynomials indexed by partitions with at most $2$ parts). Note that, if $q(z,w) = Q(z,w,zw)$, then $Q$ has total order $p$ at $(T_1,T_2,T_3) = (\zeta^{m_1}, \zeta^{m_2}, \zeta^{m_1+m_2})$ for all $m_1$,$m_2$ $\in \mathbb{Z}$ such that $(\zeta^{m_1},\zeta^{m_2})$ does not equal $(1,1)$. Thus to prove the existence of our desired $\{c_{ij}\}$ it suffices to show that $q(z,w)$ is Schur positive.

\begin{defn}
Let $f(T) = a_0 + \dotsb + a_d T^d$ be a degree $d$ polynomial with nonnegative real coefficients. We say that $f(T)$ is \emph{log-concave} if ${a_i}^2 \geq a_{i-1}a_{i+1}$ for all $i = 1, \dotsc, d-1$. We say that $f(T)$ \emph{has no internal zeros} if there do not exist integers $0 \leq i < j < k \leq d$ satisfying $a_i \neq 0$, $a_j = 0$, and $a_k \neq 0$. 
\end{defn}

\begin{lem} \textnormal{(c.f. \cite{Stan}, Proposition $1$)} \label{lem: Log-concave and Schur}
    Let $f(T) \in \mathbb{R}[T]$ be a real polynomial, so $p(z,w) := f(z)f(w)$ is a symmetric polynomial in the two variables $z$ and $w$. Then $p(z,w)$ is Schur positive if and only if $f(T)$ is log-concave and has no internal zeros.
\end{lem}

\begin{proof}

We write
\begin{equation*}
    f(T) = a_0 + a_1T + \dotsb + a_dT^d.        
\end{equation*}
Then,
\begin{equation*}
    p(z,w) = \left ( a_0 + a_1z + \dotsb + a_dz^d \right ) \left ( a_0 + a_1w + \dotsb + a_dw^d \right ) = \sum_{r=0}^{2d} \sum_{i+j = r} a_ia_j z^iw^j.
\end{equation*}
We will show that $\sum_{i+j = r} a_ia_j z^i w^j$ is Schur positive for all $r = 0, \dotsc, 2d$ if and only if $f(T)$ is log-concave and has no internal zeros.

Let $r \in \{0, \dotsc, 2d\}$ and set $m = \lfloor r/2 \rfloor$. Since $p(z,w)$ is symmetric in $z$ and $w$, we can write
\begin{equation*}
\sum_{i+j = r} a_ia_j z^i w^j = \sum_{\ell=0}^{m} (a_\ell a_{r-\ell} - a_{\ell-1}a_{r-\ell+1}) (z^\ell w^{r-\ell} + \dotsb + z^{r-\ell} w^\ell),
\end{equation*}
where $a_{-1} = 0$. Additionally,
\begin{equation*}
\sum_{\ell=0}^{m} (a_\ell a_{r-\ell} - a_{\ell-1}a_{r-\ell+1}) (z^{r-\ell} w^{\ell} + \dotsb + z^{\ell} w^{r-\ell}) = \sum_{\ell=0}^{m} (a_\ell a_{r-\ell} - a_{\ell-1}a_{r-\ell+1}) S_{(r-\ell, \ell)}
\end{equation*}
by \Cref{lem: Schur in 2 vars}. Consequently, $p(z,w)$ is Schur positive if and only if $a_ia_j \geq a_{i-1}a_{j+1}$ whenever $i \leq j$. We claim that $a_ia_j \geq a_{i-1}a_{j+1}$ whenever $i \leq j$ if and only if $f(T)$ is log-concave with no internal zeros.

Suppose that $a_i a_j \geq a_{i-1} a_{j+1}$ whenever $0 \leq i \leq j \leq 2d$. In particular, ${a_i}^2 \geq a_{i-1} a_{i+1}$ for all $1 \leq i \leq 2d-1$. Suppose for contradiction that $f(T)$ has an internal zero. Then there exists $0 \leq i < j \leq 2d$ such that $a_i \neq 0$ and $a_j \neq 0$ but $a_k = 0$ for all $i < k < j$. Then, 
\begin{equation*}
0 = a_{i+1}a_{j-1} \geq a_ia_j,
\end{equation*}
which contradicts the fact that $f(T)$ has nonnegative coefficients.

Conversely, suppose that $f(T)$ is log-concave with no internal zeros. Let $i, j \in \{0, \dotsc, 2d\}$ be two elements such that $i \leq j$. If either $a_{i-1} = 0$ or $a_{j+1} = 0$, then $a_ia_j \geq a_{i-1} a_{j+1} = 0$. Otherwise, $a_{i-1} \neq 0$ and $a_{j+1} \neq 0$. Then, the ratio $\frac{a_k}{a_{k+1}}$ is well-defined for all $k = i-1, \dotsc, j$ because $f(T)$ has no internal zeros and $\frac{a_j}{a_{j+1}} \geq \frac{a_{i-1}}{a_{i}}$ because $f(T)$ is log-concave.

\end{proof}

\begin{prop}\label{prop:explicitconstants}
The polynomial 
\begin{equation*}
q(z,w) = (1+z+ \dotsb + z^d)^p(1 + w + \dotsb + w^d)^p
\end{equation*}
is Schur-positive. In other words, there exist nonnegative constants $\{c_{ij}\}_{i,j \geq 0}$ (and only finitely many of them are nonzero) with
\begin{equation*}
q(z,w) = \sum_{i,j \geq 0} c_{ij} S_{(i+j,j)}.
\end{equation*}
In fact, 
\begin{equation*}
c_{ij} =a_{i+j} a_j - a_{i+j+1}a_{j-1} 
\end{equation*}
where the $a_i$ are defined by
$$\sum_i a_i t^i = (1+ t + \cdots t^d)^p.$$
\end{prop}

\begin{rem}
The $c_{ij}$ are non-negative since the $a_{i}$ are log-concave.
\end{rem}


\begin{proof}

Since $f(T) = 1 + T + \dotsb + T^d$ is log-concave with no internal zeros, \Cref{lem: Log-concave and Schur} implies that 
\begin{equation*}
(1+z+ \dotsb + z^d)(1 + w + \dotsb + w^d)
\end{equation*}
is Schur positive. As products of Schur positive polynomials are Schur positive, $q(z,w)$ is Schur positive.

Next, we compute the coefficients of $q(z,w)$. We observe that
\begin{equation*}
\begin{split}
(1 + z + \dotsb + z^d)^p & = \frac{(1-z^{d+1})^p}{(1-z)^p} \\
& = \left ( \sum_{k=0}^p (-1)^k \binom{p}{k} z^{(d+1)k} \right ) \left ( \sum_{k = 0}^\infty \binom{k+p-1}{p-1} z^k \right ) \\
& = \sum_{k=0}^{dp} \left ( \sum_{\ell = 0}^{\lfloor k/(d+1) \rfloor } (-1)^\ell \binom{p}{\ell} \binom{k +p -  (d+1)\ell - 1 }{p-1}\right ) z^k.
\end{split}
\end{equation*}
Set
\begin{equation*}
a_k = 
\begin{cases}
\sum_{\ell = 0}^{\lfloor k/(d+1) \rfloor } (-1)^\ell \binom{p}{\ell} \binom{k +p -  (d+1)\ell - 1 }{p-1} & \text{if $0 \leq k \leq dp$ and} \\
0 & \text{otherwise.}
\end{cases}
\end{equation*}
Then,
\begin{equation*}
\begin{split}
(1 + z + \dotsb + z^d)^p (1 + w + \dotsb + w^d)^p & = \left ( \sum_{k} a_k z^k \right ) \left ( \sum_k a_k y^k \right ) \\
& = \sum_{r = 0}^{\infty} \sum_{i+j = r} a_i a_j z^i y^j.
\end{split}
\end{equation*}
For any $r \geq 0$, we can use \Cref{lem: Schur in 2 vars} to write
\begin{equation*}
\sum_{i+j = r} a_i a_j z^i y^j = \sum_{i = \lceil r/2 \rceil}^{r} (a_ia_{r-i} - a_{i+1} a_{r-i-1}) S_{(i, r-i)}. 
\end{equation*}
As a result,
\begin{equation*}
\begin{split}
\sum_{r = 0}^{\infty} \sum_{i+j = r} a_i a_j z^i y^j & = \sum_{r = 0}^{\infty} \sum_{i = \lceil r/2 \rceil}^{r} (a_ia_{r-i} - a_{i+1} a_{r-i-1}) S_{(i, r-i)} \\
& = \sum_{i,j \geq 0} (a_{i+j}a_j - a_{i+j+1}a_{j-1}) S_{(i+j,j)}.
\end{split}
\end{equation*}
Substituting $a_k = 0$ for all $k < 0$ and $k > dp$, we see that
\begin{equation*}
\begin{split}
q(z,w) = & \sum_{i =0}^{dp - 1} a_i a_0 S_{(i,0)} \\ 
& + \sum_{i + j = dp}  a_{dp} a_{j} S_{(dp,j)} \\
& + \sum_{i+j < dp, j > 0} (a_{i+j} a_{j} - a_{i+j+1}a_{j-1}) S_{(i+j,j)}.
\end{split}
\end{equation*}

\end{proof}

\newpage

\bibliographystyle{habbrv}

\bibliography{refs}

\end{document}